\documentclass[reqno, 11pt]{article}

\usepackage[square,numbers]{natbib}
\RequirePackage[colorlinks,citecolor=blue,urlcolor=blue]{hyperref}
\usepackage{mathrsfs,amsthm,amsmath}
\usepackage{amssymb,multirow}
\usepackage{amsfonts}
\usepackage{stmaryrd}
\usepackage{dsfont}

\usepackage{ucs}
\usepackage[utf8x]{inputenc}
\usepackage{amsmath}
\usepackage{amsfonts}
\usepackage{amssymb}
\usepackage{amsthm}
\usepackage{fontenc}
\usepackage{graphicx}

\usepackage{bbm}
\usepackage{bm}
\RequirePackage[OT1]{fontenc}

\newcommand{\R}{\mathds{R}}

\newcommand{\C}{\mathds{C}}
\newcommand{\N}{\mathds{N}}

\newcommand{\eesf}{\mathsf{E}}
\newcommand{\ffsf}{\mathsf{F}}
\newcommand{\ggsf}{\mathsf{G}}

\newcommand{\ppsf}{\mathsf{P}}

\newcommand{\Ccr}{\mathscr{C}}

\newcommand{\ddr}{\mathrm{d}}

\newcommand{\Ical}{\mathcal{I}}

\newcommand{\ind}{\mathds{1}}

\theoremstyle{plain}
\newtheorem{thm}{\textsc{Theorem}}

\newtheorem{lem}{\textsc{Lemma}}
\newtheorem{cor}{\textsc{Corollary}}
\newtheorem{rmk}{Remark}
\newtheorem{prp}{\textsc{Proposition}}

\allowdisplaybreaks

\makeatletter
\newcommand{\subjclass}[2][2010]{%
  \let\@oldtitle\@title%
  \gdef\@title{\@oldtitle\footnotetext{#1 \emph{Mathematics subject classification.} #2}}%
}
\newcommand{\keywords}[1]{%
  \let\@@oldtitle\@title%
  \gdef\@title{\@@oldtitle\footnotetext{\emph{Key words and phrases.} #1.}}%
}
\makeatother

\author{Emanuele Dolera and Stefano Favaro}

\title{\textbf{A Berry-Esseen theorem for Pitman's $\alpha$-diversity}}



\date{}

\begin{document}
\maketitle

\begin{abstract}
This paper contributes to the study of the random number $K_n$ of blocks in the random partition of $\{1,\ldots,n\}$ induced by random sampling from the two parameter Poisson-Dirichlet process. For any $\alpha\in(0,1)$ and $\theta>-\alpha$ Pitman \cite{Pit(06)} showed that $n^{-\alpha}K_n\stackrel{\text{a.s.}}{\longrightarrow} S_{\alpha,\theta}$ as $n\rightarrow+\infty$, where $S_{\alpha,\theta}$, is known  as Pitman's $\alpha$-diversity. Our main result is a Berry-Esseen theorem for Pitman's $\alpha$-diversity
$S_{\alpha,\theta}$, namely we show that
\begin{displaymath}
\sup_{x \geq 0} \left| \ppsf\left[\frac{K_n}{n^{\alpha}} \leq x \right] - \ppsf[S_{\alpha,\theta} \leq x] \right| \leq \frac{C(\alpha, \theta)}{n^{\alpha}}
\end{displaymath}
holds for every $n \in \N$ with an explicit constant term $C(\alpha, \theta)$, for $\alpha\in(0,1)$ and $\theta>0$. The proof relies on three intermediate novel results which are of independent interest: i) a representation of the law of $K_n$ as a compound distribution; ii) a quantitative version of Laplace's approximation method; iii) a refined quantitative bound for Poisson approximation. An application of our Berry-Esseen theorem is presented in the context of Bayesian nonparametric inference for species sampling problems.
\end{abstract}


\section{Introduction}

The two parameter Poisson-Dirichlet process is a discrete (almost surely) random probability measure introduced by Perman et al. \cite{Per(92)}. For any $\alpha\in[0,1)$ and $\theta>-\alpha$, let $\{V_i\}_{i\geq1}$ be independent random variables such that $V_i$ is distributed as a Beta distribution with parameter $(1-\alpha,\theta+i\alpha)$, for $i\geq1$. If $P_1 := V_1$ and $P_i := V_i\prod_{1\leq j\leq i-1}(1-V_j)$ for $i \geq 2$, then $\sum_{i\geq1}P_i=1$ almost surely. The two parameter Poisson-Dirichlet process is defined as the random probability measure $\tilde{\mathfrak{p}}_{\alpha,\theta}$ on $(\N, 2^{\N})$ such that $\tilde{\mathfrak{p}}_{\alpha,\theta}(\{i\}) = P_i$ for all $i \in \N$. A random sample $(X_1,\ldots,X_n)$ from $\tilde{\mathfrak{p}}_{\alpha,\theta}$ is the first $n$-segment of the $\N$-valued exchangeable sequence $\{X_i\}_{i \geq 1}$ having $\tilde{\mathfrak{p}}_{\alpha,\theta}$ as directing measure. Due to the discreteness of $\tilde{\mathfrak{p}}_{\alpha,\theta}$, $(X_1,\ldots,X_n)$ induces a random partition $\Pi_{n}$ of $\{1,\ldots,n\}$ by means of the (random) equivalence relation $i\sim j\iff X_i = X_j$ (Aldous \cite{Ald(85)} and Pitman \cite{Pit(06)}). Let $K_{n}:=K_{n}(X_{1},\ldots,X_{n})$, where $K_{n}(X_{1},\ldots,X_{n})$ is the (random) number of blocks of $\Pi_{n}$, i.e., the random number of distinct elements in $(X_1,\ldots,X_n)$. Let $N_{j,n}:=N_{j,n}(X_{1},\ldots,X_{n})$, where $N_{j,n}(X_{1},\ldots,X_{n})$ is the (random) size of the $j$-th block of $\Pi_{n}$. If $\mathbf{N}_{n}:=(N_{1,n},\ldots,N_{K_{n},n})$, then Pitman \cite{Pit(95)} (see also Lijoi et al. \cite{Lij(14)}) showed that
\begin{equation}\label{esf}
\ppsf[K_{n}=j,\mathbf{N}_n =(n_1,\ldots,n_j)]=\frac{1}{j!}{n\choose n_{1},\ldots,n_{j}}\frac{[\theta]_{(j,\alpha)}}{[\theta]_{(n)}}\prod_{i=1}^j [1-\alpha]_{(n_i-1)},
\end{equation}
for any $j \in \{1, \dots, n\}$ and $(n_1,\ldots,n_j)\in\N^{j}$ such  that $\sum_{1\leq i\leq j} n_i = n$, where $[x]_{(n,a)}$ denotes the rising factorial of $x$ of order $n$ and increment $a$, i.e. $[x]_{(n,a)}:=\prod_{0\leq i\leq n-1 }(x+ia)$, and $[x]_{(n)} := [x]_{(n,1)}$. Equation \eqref{esf} is referred to as Ewens-Pitman sampling formula. The two parameter Poisson-Dirichlet process plays a fundamental role in a variety of research areas, such as mathematical population genetics, Bayesian nonparametric statistics, statistical machine learning, excursion theory, combinatorics and statistical physics. We refer to Pitman \cite{Pit(06)} and Crane \cite{Cra(16)} for a comprehensive treatment.

There have been several studies on the large $n$ behaviour of $K_n$. For $\alpha=0$ and $\theta > 0$, $\tilde{\mathfrak{p}}_{0,\theta}$ reduces to the Dirichlet process and
$K_n =\sum_{1\leq i\leq n} Z_i$ where the $Z_i$'s are independent Bernoulli random variables  with parameter $\theta/(\theta+i-1)$, for $i=1,\ldots,n$. Korwar and Hollander \cite{Kor(73)} showed that $K_n/\log(n)$ converges almost surely to $\theta$ as
$n\rightarrow+\infty$. Also, it follows from Lindberg's theorem that $(K_n-\theta\log(n))/\sqrt{\theta\log(n)}$ converges weakly to a standard Gaussian random variable as $n\rightarrow+\infty$. For $\alpha\in(0,1)$ the large $n$ Gaussian limit for $K_n$ no longer holds. In particular, Theorem 3.8 in Pitman \cite{Pit(06)} exploited a martingale convergence argument to show that 
\begin{equation}\label{adiv}
\frac{K_n}{n^{\alpha}}\stackrel{\text{a.s.}}{\longrightarrow} S_{\alpha,\theta}
\end{equation}
as $n\rightarrow+\infty$, where $0<S_{\alpha,\theta}<+\infty$ is a random variable distributed as a scaled Mittag-Leffler distribution. Precisely, let $\Gamma$ stand for the Gamma function, and let 
\begin{equation} \label{humbert}
f_{\alpha}(z) = \frac{1}{\pi} \sum_{j\geq1}\frac{(-1)^{j+1}}{j!} \sin(\pi\alpha j) \frac{\Gamma(\alpha j + 1)}{z^{\alpha j + 1}}\ind\{z > 0\},
\end{equation}
for $\alpha\in(0,1)$, be the positive $\alpha$-stable density function. Then, $S_{\alpha,\theta}$ has density function
\begin{equation}\label{ml}
f_{S_{\alpha,\theta}}(s) = \frac{\Gamma(\theta+1)}{\alpha\Gamma(\theta/\alpha+1)}s^{\frac{\theta-1}{\alpha} -1} f_{\alpha}(s^{-1/\alpha})\ind\{s > 0\}.
\end{equation}
One may easily generate random variates from $S_{\alpha,\theta}$ (e.g., Devroye \cite{Dev(09)}). For $\theta=0$, equation \eqref{ml} reduces to the Mittag-Leffler density function. For $\alpha\in(0,1)$ and $\theta>-\alpha$, the random variable $S_{\alpha,\theta}$ is referred to as Pitman's $\alpha$-diversity. Large and moderate deviations for $K_n$ are established in Feng and Hoppe \cite{Fen(98)} and Favaro et al. \cite{Fav(18)}, whereas a concentration inequality for $K_n$ is obtained in Pereira et al. \cite{Per(18)} by relying on concentration inequalities for martingales.

In this paper, we formulate a Berry-Esseen theorem for Pitman's $\alpha$-diversity $S_{\alpha,\theta}$. In particular, let $\ffsf_n$ and $\ffsf_{\alpha,\theta}$ stand for the distribution functions of $K_n/n^{\alpha}$ and $S_{\alpha,\theta}$, respectively, i.e. $\ffsf_n(x) := \ppsf[K_n/n^{\alpha} \leq x]$ and $\ffsf_{\alpha,\theta}(x) := \ppsf[S_{\alpha,\theta} \leq x]$, for any $x > 0$. To measure the discrepancy between  $\ffsf_n$ and $\ffsf_{\alpha,\theta}$, we consider the Kolmogorov distance which, for any pair of distribution functions $\ffsf_1$ and $\ffsf_2$ supported in $[0,+\infty)$, is defined as $\ddr_K(\ffsf_1; \ffsf_2) := \sup_{x \geq 0} |\ffsf_1(x) - \ffsf_2(x)|$. The next theorem states our Berry-Esseen bound.
\begin{thm}\label{be}
For any $\alpha\in(0,1)$ and $\theta>0$, there exists a positive constant $C_{\alpha, \theta}$, depending solely on $\alpha$ and $\theta$, such that $\ddr_K(\ffsf_n; \ffsf_{\alpha,\theta}) \leq n^{-\alpha}C_{\alpha,\theta}$ for every $n \in \N$.
\end{thm}
Theorem \ref{be} is the first quantitative version of Theorem 3.8 in Pitman \cite{Pit(06)}. The constant $C_{\alpha, \theta}$ is obtained constructively, and its value can be made explicit by gathering equations displayed in Section 2. 
The proof of Theorem \ref{be} relies on three intermediate novel results which are of independent interest: i) a probabilistic representation of the distribution of $K_{n}$ as a compound distribution where the latent term is the distribution of the random number of blocks in Poisson compound random partition model, and the mixing term is the law of a scale mixture between Pitman's $\alpha$-diversity and a Gamma random variable; ii) a quantitative version of the asymptotic expansion, in the sense of Poincar\'e, of a Laplace-type integral; iii) a refined quantitative bound for Poisson approximation which improves results in Hwang \cite{Hwa(99)}.

We present an application of Theorem \ref{be} in Bayesian nonparametric inference for species sampling. Consider a population of individuals belonging to an infinite number of species with unknown proportions. Given an initial (observable) random sample of size $n$ from the population, a classical statistical problem is to infer the number $K_{m}^{(n)}$ of hitherto unseen species that would be observed in $m$ additional (unobservable) samples (Orlitsky et al. \cite{Orl(16)}). A Bayesian approach to this problem was proposed in Lijoi et al.
\cite{Lij(07)}, and it relies on the law of $\tilde{\mathfrak{p}}_{\alpha,\theta}$ as a prior distribution for the unknown species composition of the population. That is, the Ewens-Pitman sampling formula \eqref{esf} models the species composition of the initial (observable) random sample $(X_{1},\ldots,X_{n})$ from the population, i.e., the number $K_{n}$ of species and their frequencies $\mathbf{N}_{n}$. Lijoi et al. \cite{Lij(07)} first derived the posterior distribution, given $\{K_{n},\mathbf{N}_{n}\}$, of $K_{m}^{(n)}$. Then, Favaro et al. \cite{Fav(09)} showed that there exists a random variable $S_{\alpha,\theta}(n,j)$, referred to as Pitman's posterior $\alpha$-diversity, such that
\begin{equation}\label{post_adiv}
\ppsf\left[ \frac{K_{m}^{(n)}}{m^{\alpha}} \longrightarrow S_{\alpha,\theta}(n,j)\ \text{as}\ m\rightarrow+\infty\,\Big|\, K_n = j,\mathbf{N}_{n}=(n_{1},\ldots,n_{j}) \right] = 1
\end{equation}
holds for any $j \in \{1, \dots, n\}$ and $(n_1,\ldots,n_j) \in \N^{j}$ with $\sum_{1\leq i\leq j} n_i = n$. Also, $\ppsf[S_{\alpha,\theta}(n,j) \in \cdot |\, 
K_n = j,\mathbf{N}_{n}=(n_{1},\ldots,n_{j})] = \ppsf[B_{j+\theta/\alpha,n/\alpha-j}S_{\alpha,\theta+n} \in \cdot]$ where, under $\ppsf$: i) $S_{\alpha,\theta+n}$ has a density given by \eqref{ml}; ii) $B_{j+\theta/\alpha,n/\alpha-j}$ is a Beta random variable with parameter $(j+\theta/\alpha,n/\alpha-j)$; iii) $S_{\alpha,\theta+n}$ and $B_{j+\theta/\alpha,n/\alpha-j}$ are independent. The importance of \eqref{post_adiv} is motivated by the fact that the computational burden for evaluating the posterior distribution of $K_m^{(n)}$ becomes overwhelming for large $m$. This is common in genomics, where DNA libraries consists of millions of genes. In such a context,  \eqref{post_adiv} has been extensively applied to obtain large $m$ approximated posterior inferences for $K_{m}^{(n)}$ (Favaro et al. \cite{Fav(09)}). We show that Theorem \ref{be} leads to formulate a Berry-Esseen theorem for Pitman's posterior $\alpha$-diversity, thus quantifying the error in approximating the posterior distribution of $K_{m}^{(n)}$ with the law of $S_{\alpha,\theta}(n,j)$. 

The paper is structured as follows. In Section 2 we prove Theorem \ref{be}, whereas in Section 3 we state and prove a Berry-Esseen theorem for Pitman's posterior $\alpha$-diversity.


\section{Proof of Theorem \ref{be}}

The proof is divided into four main parts, developed in the Subsections \ref{sect:representation}---\ref{sect:conclusion}. Along these subsections, we make use of the notion of probability generating function (PGF) of a random variable $X$ with values in $\N_0 := \{0, 1, 2, \ldots\}$, i.e. $\ggsf_X(s) := \sum_{x \geq 0} \ppsf[X = x] s^x$.

\subsection{A new probabilistic representation for $K_n$} \label{sect:representation}

Consider a population of individuals containing a random number $N_{\lambda}$ of types, where $N_{\lambda}$ has a Poisson distribution with
parameter $\lambda = z[1 - (1-q)^{\alpha}]$, for $\alpha\in(0,1)$, $q\in(0,1)$ and $z>0$. For any $j\in\N$, let $Q_j(\alpha,q)$ be the random number of individuals of type $j$ in the population. Assume the $Q_j(\alpha,q)$'s are independent of $N_{\lambda}$ and independent of each other, with the same probability law
\begin{equation}\label{zero_tru}
\ppsf[Q_1(\alpha,q) = x] = -\frac{1}{[1 - (1-q)^{\alpha}]}{\alpha\choose x}(-q)^x = \frac{\Gamma(\alpha+1) \sin\pi\alpha }{\pi [1 - (1-q)^{\alpha}]} \frac{\Gamma(x-\alpha)}{x!} q^x 
\end{equation}
for any $x \in \N$. In the next lemma we derive the conditional distribution of $N_{\lambda}$ given the size of the random sample $S(\alpha,q,z) := \sum_{j=0}^{N_{\lambda}} Q_j(\alpha,q)$, with the proviso that $\ppsf[Q_0(\alpha,q) = 0] = 1$. This distribution relies on a noteworthy probability distribution $\rho(\cdot; \alpha,n,z)$ on $\{1, \dots, n\}$ involving the
generalized factorial coefficients, namely $\Ccr(n,k;\alpha) := \frac{1}{k!} \sum_{i=1}^k (-1)^i {k\choose i} [-i\alpha]_{(n)}$ (Charalambides \cite{Cha(02)}). 

\begin{lem}\label{lm:caralambides}
Let $\{Q_j(\alpha,q)\}_{j \geq 1}$ be i.i.d. random variables distributed according to \eqref{zero_tru}, and define
$S(\alpha,q,z) := \sum_{j=0}^{N_{\lambda}} Q_j(\alpha,q)$, with the proviso that $\ppsf[Q_0(\alpha,q) = 0] = 1$. Then, for every $n \in \N$ and $k \in \{1, \ldots, n\}$, it holds
\begin{equation} \label{cm_dist}
\rho(\{k\}; \alpha,n,z) := \ppsf[N_{\lambda} = k\,|\, S(\alpha,q,z) = n] = \frac{\Ccr(n,k;\alpha) z^k}{\sum_{j=1}^n \Ccr(n,j;\alpha) z^j} \ .
\end{equation} 
\end{lem}

\begin{proof}[Proof of Lemma \ref{lm:caralambides}]
For fixed $n \in \N$ and $k \in \{1, \ldots, n\}$, one has $\ppsf[S(\alpha,q,z) = n\,|\, N_{\lambda} = k] = \ppsf[\sum_{j=1}^k Q_j(\alpha,q) = n]$. By virtue of the binomial series, 
the PGF $\ggsf(\cdot; k,\alpha,q)$ of $\sum_{j=1}^k Q_j(\alpha,q)$ reads
$$
\ggsf(s; k,\alpha,q) = \left\{-\sum_{x=1}^{\infty} \frac{1}{[1 - (1-q)^{\alpha}]}{\alpha\choose x}(-sq)^x \right\}^k 
= \Big[\frac{1 - (1-sq)^{\alpha}}{1 - (1-q)^{\alpha}}\Big]^k
$$
for $|s|<1$. Since $[1 - (1-u)^{\alpha}]^k = k! \sum_{n \geq k} \Ccr(n,k;\alpha) \dfrac{u^n}{n!}$ holds whenever $|u| < 1$ (see Theorem 8.14 in Charalambides \cite{Cha(02)}), conclude that
\begin{equation}\label{cm_distBayes}
\ppsf[S(\alpha,q,z) = n\,|\, N_{\lambda} = k] = \frac{k!}{[1 - (1-q)^{\alpha}]^k} \Ccr(n,k;\alpha) \frac{q^n}{n!}
\end{equation}
for $n \geq k$. Moreover, \eqref{cm_distBayes} holds also for $k>n$ since $\Ccr(n,k;\alpha) = 0$. Hence, using the explicit expression of $\ppsf[N_{\lambda} = k]$, \eqref{cm_dist} ensues from \eqref{cm_distBayes} by means of the Bayes formula.
\end{proof}

Let $G_{\tau,\lambda}$ be a Gamma random variable with scale parameter 
$\lambda > 0$ and shape parameter $\tau > 0$, and let $\{R(\alpha,n,z)\}_{z > 0}$ be a family of random variables such that $\ppsf[R(\alpha,n,z)=k] = \rho(\{k\}; \alpha,n,z)$, for any $n \in \N$, $k \in \{1, \ldots, n\}$, $\alpha \in (0,1)$ and $z>0$.

\begin{prp}\label{prop:representation_prior}
For fixed $n \in \N$, $\alpha \in (0,1)$ and $\theta>-\alpha$, there holds the following (distributional) identity
\begin{equation}\label{representation_prior}
K_n \stackrel{\text{d}}{=} R(\alpha, n, S_{\alpha,\theta} G_{\theta+n,1}^{\alpha}),
\end{equation}
where $S_{\alpha,\theta}$, $G_{\theta+n,1}$ and $\{R(\alpha,n,z)\}_{z > 0}$ are independent random elements.
\end{prp}

\begin{proof}[Proof of Proposition \ref{prop:representation_prior}]
Start from the well-known identity (e.g., Pitman \cite{Pit(06)})
\begin{equation}\label{dist_prior}
\ppsf[K_n=k] = \frac{[\theta]_{(k,\alpha)}}{[\theta]_{(n)}} \frac{\Ccr(n,k;\alpha)}{\alpha^k} = \frac{\Gamma(\theta/\alpha + k)}{\Gamma(\theta/\alpha)} \frac{\Gamma(\theta)}{\Gamma(\theta + n)} \Ccr(n,k;\alpha)
\end{equation}
for any $k \in \{1, \ldots, n\}$. Due to the identity $\int_{0}^{+\infty} x^{-\theta} f_{\alpha}(x) \ddr x = \dfrac{\Gamma(\theta/\alpha)}
{\alpha \Gamma(\theta)}$, it is easily checked that 
\begin{equation} \label{Referee5}
f(z; \alpha, \theta, n) = \frac{z^{\theta/\alpha+n/\alpha-1}}{\Gamma(\theta/\alpha) [\theta]_{(n)}} \Big(\int_{0}^{+\infty}x^n e^{-xz^{1/\alpha}} f_{\alpha}(x)\ddr x\Big) \ind\{z > 0\}
\end{equation}
is a probability density function. Thus, one has the following identities
\begin{align*}
&\ppsf[K_n=k]\\
&\quad= \frac{\Ccr(n,k;\alpha)}{\Gamma(\theta/\alpha) [\theta]_{(n)}} \int_0^{+\infty} \!\!\! z^{k+\theta/\alpha-1} e^{-z}\ddr z\\
&\quad=\frac{1}{\Gamma(\theta/\alpha) [\theta]_{(n)}} \int_0^{+\infty}\!\!\! z^{\theta/\alpha-1} e^{-z}
\Big(\sum_{j=1}^n \Ccr(n,j;\alpha) z^j \Big) \frac{\Ccr(n,k;\alpha) z^k}{\sum_{j=1}^n \Ccr(n,j;\alpha)z^j}\ddr z \\
&\quad=\frac{1}{\Gamma(\theta/\alpha) [\theta]_{(n)}} \int_0^{+\infty} \!\!\! z^{\theta/\alpha+n/\alpha-1}
\Big(\int_0^{+\infty} \!\!\! x^n e^{-xz^{1/\alpha}} f_{\alpha}(x) \ddr x\Big) \frac{\Ccr(n,k;\alpha) z^k}{\sum_{j=1}^n \Ccr(n,j;\alpha) z^j} \ddr z\\
&\quad=\int_0^{+\infty} \ppsf[R(\alpha,n,z) = k] f(z; \alpha, \theta, n)\ddr z,
\end{align*}
where: i) the first identity follows from \eqref{dist_prior} upon noticing that $\alpha^{-k}[\theta]_{(k,\alpha)} = \Gamma(k+\theta/\alpha)/\Gamma(\theta/\alpha)$; ii) the third exploits the following identity
\begin{equation} \label{bernardo}
\sum_{j=1}^n \Ccr(n,j;\alpha) z^j  = e^z z^{n/\alpha}\int_0^{+\infty} \!\!\! x^n e^{-xz^{1/\alpha}} f_{\alpha}(x) \ddr x
\end{equation} 
displayed in Proposition 1 of Favaro et al. \cite{Fav(15)}; the fourth follows from a combination of \eqref{cm_dist} with \eqref{Referee5}. To conclude, it is enough to show that the probability distribution of $S_{\alpha,\theta} G_{\theta+n,1}^{\alpha}$ possesses a density coinciding with $f(\cdot; \alpha, \theta, n)$. In fact, one has 
\begin{align*}
&\ppsf[S_{\alpha,\theta} G_{\theta+n,1}^{\alpha} \leq u]\\
&\quad= \int_0^{+\infty} \!\!\! \ppsf\big[G_{\theta+n,1} \leq \big(\frac{u}{s}\big)^{1/\alpha}\big] \frac{\Gamma(\theta+1)}{\alpha\Gamma(\theta/\alpha+1)} s^{\frac{\theta-1}{\alpha} -1} f_{\alpha}(s^{-1/\alpha})\ddr s\\
&\quad= \int_0^{+\infty} \!\!\!  \Big(\int_0^{u^{1/\alpha}x}  \frac{t^{\theta+n-1}e^{-t}}{\Gamma(\theta+n)} \ddr t\Big) \frac{\Gamma(\theta+1)}{\Gamma(\theta/\alpha+1)} x^{-\theta} f_{\alpha}(x)\ddr x\\
&\quad= \int_0^{+\infty} \!\!\! \Big(\int_0^u  \frac{1}{\Gamma(\theta+n)} e^{-xz^{1/\alpha}} z^{\frac{\theta+n}{\alpha} -1}  \ddr z\Big) \frac{\Gamma(\theta+1)}{\alpha\Gamma(\theta/\alpha+1)} x^n f_{\alpha}(x)\ddr x
\end{align*}
where: i) the first identity follows from conditioning; ii) the second and the third ensue from the changes of variable $x = s^{-1/\alpha}$ and $t= xz^{1/\alpha}$, respectively.
\end{proof}

Observe that the representation \eqref{representation_prior} highlights the central role of the probability distribution $\rho(\cdot; \alpha,n,z)$. The next result describes the asymptotic behavior of such distribution for large values of $n$, to be used later on.

\begin{lem}\label{lm:asymp_compound}
For fixed $\alpha \in (0,1)$ and $z > 0$, $\lim_{n\rightarrow+\infty} \rho(\{k\}; \alpha,n,z) = e^{-z} \frac{z^{k-1}}{(k-1)!}$ for any $k \in \N$. This is tantamount to 
saying that $R(\alpha,n,z) \longrightarrow1+N_z$ in distribution, as $n\rightarrow+\infty$.
\end{lem}

\begin{proof}[Proof of Lemma \ref{lm:asymp_compound}]
Letting $\ggsf(\cdot; \alpha,n,z)$ be the PGF of the random variable $R(\alpha,n,z)$, we show that $\ggsf(s; \alpha,n,z) \rightarrow s\exp\{z(s-1)\}$ as 
$n\rightarrow+\infty$, for any $s > 0$. By using the definition of $\Ccr(n,k;\alpha)$, for every $u \in \C$ we write
\begin{align*}
\sum_{k=1}^n \Ccr(n,k;\alpha) u^k &= \sum_{i=1}^n (-1)^i [-i\alpha]_{(n)} \sum_{k=i}^n \frac{1}{k!} {k\choose i} u^k \\
&= \sum_{i=1}^n (-1)^i [-i\alpha]_{(n)} e^u u^i \frac{\Gamma(n-i+1,u)}{i!\Gamma(n-i+1)}
\end{align*}
where $\Gamma(a,x) := \int_x^{+\infty} t^{a-1} e^{-t} \ddr t$ denotes the incomplete gamma function for $a,x > 0$. Whence, $\ggsf(\cdot; \alpha,n,z)$ has the following expression
$$
\ggsf(s; \alpha,n,z) = e^{z(s-1)} \frac{-zs \frac{\Gamma(n,zs)}{\Gamma(n)} + \sum_{i=2}^n (-1)^i \frac{[-i\alpha]_{(n)}}{[-\alpha]_{(n)}} (zs)^i \frac{\Gamma(n-i+1,zs)}{i!\Gamma(n-i+1)}}
{-z \frac{\Gamma(n,z)}{\Gamma(n)} + \sum_{i=2}^{n} (-1)^i \frac{[-i\alpha]_{(n)}}{[-\alpha]_{(n)}} z^i \frac{\Gamma(n-i+1,z)}{i!\Gamma(n-i+1)}}\ .
$$
Since $\lim_{n \rightarrow +\infty} \frac{\Gamma(n,x)}{\Gamma(n)} = 1$ for any $x > 0$, the proof ends by showing that
\begin{displaymath}
\lim_{n \rightarrow +\infty} \sum_{i=2}^n (-1)^i \frac{[-i\alpha]_{(n)}}{[-\alpha]_{(n)}} \frac{\Gamma(n-i+1,t)}{\Gamma(n-i+1)} \frac{t^i}{i!} = 0
\end{displaymath}
for any $t > 0$. Upon noticing that $\frac{\Gamma(n,x)}{\Gamma(n)} \leq 1$, we have the following relations
\begin{align*}
\frac{1}{i!} \Big{|}\frac{[-i\alpha]_{(n)}}{[-\alpha]_{(n)}}\Big{|} &= \frac{1}{i!}\Big{|}\frac{\Gamma(n-i\alpha)}{\Gamma(-i\alpha)}\frac{\Gamma(-\alpha)}{\Gamma(n-\alpha)}\Big{|}\\
& = \frac{\Gamma(n-i\alpha)}{i! \Gamma(n-\alpha)} \frac{|\sin i\pi\alpha|}{\pi} \Gamma(i\alpha + 1) |\Gamma(-\alpha)| \\ 
& \leq |\Gamma(-\alpha)| \frac{\Gamma(n-i\alpha)}{\Gamma(n-\alpha)} \frac{\Gamma(i\alpha + 1)}{i!} \leq |\Gamma(-\alpha)| \frac{\Gamma(n-i\alpha)}{\Gamma(n-\alpha)}
\end{align*}
for all $n \in \N$, $x > 0$ and $i \in \{1, \dots, n\}$. Then, we can write 
\begin{align*}
 &\Big{|} \sum_{i=2}^n (-1)^i \frac{[-i\alpha]_{(n)}}{[-\alpha]_{(n)}} \frac{\Gamma(n-i+1,t)}{\Gamma(n-i+1)} \frac{t^i}{i!}\Big{|}\\
 &\quad\leq \sum_{i=2}^n \frac{t^i}{i!} \Big{|}\frac{[-i\alpha]_{(n)}}{[-\alpha]_{(n)}}\Big{|}\\
 &\quad \leq |\Gamma(-\alpha)| \sum_{i=2}^n t^i \frac{\Gamma(n-i\alpha)}{\Gamma(n-\alpha)} \leq |\Gamma(-\alpha)| \frac{\max_{i = 2, \dots, n} n^{i\alpha} \Gamma(n-i\alpha)}{\Gamma(n-\alpha)} \sum_{i=2}^n \left(\frac{t}{n^{\alpha}}\right)^i.
\end{align*}
The monotonic increasing character of the function $(0,n) \ni x \mapsto n^x \Gamma(n-x)$, due to $\psi(z) := \frac{\Gamma'(z)}{\Gamma(z)} \leq \log(z)$ for any $z>0$,
entails that $\max_{i = 2, \dots, n} n^{i\alpha} \Gamma(n-i\alpha) = n^{n\alpha} \Gamma(n-n\alpha)$. Thus, observe that $\frac{n^{n\alpha} \Gamma(n-n\alpha)}{\Gamma(n-\alpha)} \sim n^{\alpha}$ to conclude that
$$
\frac{\max_{i = 2, \dots, n} n^{i\alpha} \Gamma(n-i\alpha)}{\Gamma(n-\alpha)} \sum_{i=2}^n \left(\frac{t}{n^{\alpha}}\right)^i \sim \left(\frac{t}{n^{\alpha}}\right)^2 \frac{\left(\frac{t}{n^{\alpha}}\right)^{n-1} - 1}
{\left(\frac{t}{n^{\alpha}}\right) - 1} n^{\alpha} \sim \frac{1}{n^{\alpha}}
$$
as $n\rightarrow+\infty$, completing the proof.
\end{proof}

\begin{rmk}\label{rmk_ext}
Let $\Pi_{n}$ denote the random partition of the set $\{1,\ldots,n\}$ induced by a random sample $(X_{1},\ldots,X_{n})$ from $\tilde{\mathfrak{p}}_{\alpha,\theta}$, and let $M_{l,n}$ be the number of blocks with frequency $l$, i.e., $M_{l,n}=\sum_{1\leq j\leq K_{n}}\ind\{N_{j,n}=l\}$. It is not difficult to show that results analogous to Proposition \ref{prop:representation_prior} and Lemma \ref{lm:asymp_compound} hold true for $M_{l,n}$. In particular, let $R_{l}(\alpha,n,z)$ be a random variable whose distribution coincides with the conditional distribution of the number of $Q_{j}(\alpha,q)$'s equal to $l$, given the size of the random sample $S(\alpha,q,z)=n$. It can be shown that: i) for $l\geq1$, $n\in\N$, $\alpha\in(0,1)$ and $\theta>-\alpha$ the distribution of $M_{l,n}$ coincides with the distribution of the random variable $R_{l}(\alpha,n,S_{\alpha,\theta}G_{\theta+n,1}^{\alpha})$; ii) for  $l\geq1$, $\alpha\in(0,1)$ and $z>0$ the random variable $R_{l}(\alpha,n,z)$ converges weakly, as $n\rightarrow+\infty$, to a Poisson random variable with parameter $z\alpha[1-\alpha]_{(l-1)}/l!$. 
\end{rmk}


\subsection{A quantitative Laplace method for $\Ical_n(z)$} \label{sect:laplace}

Here, we study the Laplace integral $\Ical_n(z) :=  \int_0^{+\infty} e^{-n\phi_z(y)} f_{\alpha}(y) \ddr y$ for $z > 0$, where $\phi_z(y) := zy - \log y$. This quantity is connected with \eqref{bernardo} in view of the identity 
\begin{equation} \label{InBernardo}
d_n(x) := \sum_{j=1}^n \Ccr(n,j;\alpha) (x n^{\alpha})^j  =  e^{x n^{\alpha}} x^{n/\alpha} n^n \Ical_n\big(x^{1/\alpha}\big) 
\end{equation}
valid for all $x > 0$. As first step, after noticing that $\overline{y}(z) := 1/z$ is the only minimum point of $\phi_z(y)$, a direct application of the Laplace method (Section 7 in Chapter 3 of Olver \cite{Olv(74)}) shows that $\Ical_n(z) \sim \left(\frac{1}{z}\right)^{n+1} f_{\alpha}\left(\frac{1}{z}\right) e^{-n} \sqrt{\frac{2\pi}{n}}$ as $n \rightarrow +\infty$. However, a more precise large $n$ estimate is provided by the next
\begin{lem} \label{lm:Laplace1}
For any $n \in \N$, there exists a continuous function $\delta_n : (0,+\infty) \rightarrow (0,+\infty)$ such that
\begin{equation} \label{Laplace1}
\Ical_n(z) = \left(\frac{1}{z}\right)^{n+1} f_{\alpha}\left(\frac{1}{z}\right) e^{-n} \sqrt{\frac{2\pi}{n}}\ \Big[1 + \delta_n(z)\Big]
\end{equation}
and $|\delta_n(z)| \leq \Delta(z)/n$ for any $z > 0$, where $\Delta : (0,+\infty) \rightarrow (0,+\infty)$ is a suitable continuous function which is independent of $n$. Moreover, $\Delta$ can be chosen in such a way that $\Delta(z) = O(1)$ as $z \rightarrow 0$, and $\Delta(z)f_{\alpha}(1/z) = O(z^{-\infty})$ as $z \rightarrow +\infty$.
\end{lem}
Recall from Kanter \cite{Kan(86)} that the positive $\alpha$-stable density function \eqref{humbert} can be written as
\begin{equation} \label{kanter}
f_{\alpha}(z) = \frac{1}{\pi} \left(\frac{\alpha}{1-\alpha}\right) \left(\frac{1}{z}\right)^{\frac{1}{1-\alpha}} \int_0^{\pi} A(\varphi) \exp\left\{- \left(\frac{1}{z}\right)^{\frac{\alpha}{1-\alpha}} A(\varphi) \right\} \ddr \varphi
\end{equation}
holds for any $\alpha \in (0,1)$ and $z>0$, with $A(\varphi) := \left(\frac{\sin(\alpha\varphi)}{\sin(\varphi)}\right)^{\frac{1}{1-\alpha}} \left(\frac{\sin((1-\alpha)\varphi)}{\sin(\alpha\varphi)}\right)$. In particular, \eqref{kanter} entails
\begin{equation} \label{asymptotic_stable}
f_{\alpha}(z) \sim \frac{(\alpha/z)^{\frac{2-\alpha}{2(1-\alpha)}}}{\sqrt{2\pi\alpha(1-\alpha)}} \exp\left\{- (1-\alpha) \left(\frac{\alpha}{z}\right)^{\frac{\alpha}{1-\alpha}} \right\} 
\end{equation}
as $z \rightarrow 0$, the asymptotic relation remaining valid after differentiating both sides with respect to $z$ as many times as needed. See, e.g., Zolotarev \cite{Zol(86)} for details.

\begin{proof}[Proof of Lemma \ref{lm:Laplace1}]
The change of variables $s = zy -1$ gives $\Ical_n(z) =  \left(\frac{1}{z}\right)^{n+1} e^{-n} \int_{-1}^{+\infty} e^{-n h(s)} f_{\alpha}\left(\frac{s+1}{z}\right) \ddr s$, with $h(s) := s - \log(s+1)$. In order to exploit the analyticity of the function $h$ for $s \in (-1,1)$, fix $\sigma \in (0,1)$ and split the above integral into the regions 
$s \in (\sigma,+\infty)$, $s \in (0,\sigma)$, $s \in (-\sigma, 0)$ and $s \in (-1,-\sigma)$. The ensuing analysis will provide the desired bound on $\Delta$ by expressing it as a sum of three other functions 
called $\Delta_1$, $\Delta_2$ and $\Delta_3$. First, write $h(s) \geq h'(\sigma)(s-\sigma) + h(\sigma)$ for every $s \in (\sigma,+\infty)$ by the convexity of $h$, yielding
\begin{align*}
& \int_{\sigma}^{+\infty} e^{-n h(s)} f_{\alpha}\left(\frac{s+1}{z}\right) \ddr s \\
&\quad\leq (\sigma+1)^n \exp\Big\{-\frac{n\sigma}{\sigma+1}\Big\} \int_{\sigma}^{+\infty} \exp\Big\{-\frac{ns\sigma}{\sigma+1}\Big\} f_{\alpha}\left(\frac{s+1}{z}\right) \ddr s \\
&\quad= (\sigma+1)^n \int_{\sigma+1}^{+\infty} \exp\Big\{-\frac{nt\sigma}{\sigma+1}\Big\} f_{\alpha}\left(\frac{t}{z}\right) \ddr t\ .
\end{align*}
Observe that the analysis of this term leads to the study of
$$
\Delta_1(z) := \frac{1}{f_{\alpha}(1/z)} \sup_{n \in \N} n^{3/2} (\sigma+1)^n \int_{\sigma+1}^{+\infty} \exp\big\{-\frac{nt\sigma}{\sigma+1}\big\} f_{\alpha}\left(\frac{t}{z}\right) \ddr t\ .
$$
For small $z$, there holds $f_{\alpha}(1/z) = O(z^{1+\alpha})$ in view of \eqref{humbert}, and the supremum turns out to be of the same order $z$, i.e. $O(z^{1+\alpha})$,
since $f_{\alpha}(x) \leq C_{\alpha}x^{-(1+\alpha)}$. For large $z$, put $\lambda := \frac{\sigma}{2(\sigma + 1)}$ and use the Cauchy-Schwartz inequality to obtain
\begin{align*}
& \int_{\sigma+1}^{+\infty} \exp\Big\{-\frac{nt\sigma}{\sigma+1} + \lambda t\Big\} e^{-\lambda t} f_{\alpha}\left(\frac{t}{z}\right) \ddr t \\
&\quad\leq \Big(\int_{\sigma+1}^{+\infty} \exp\Big\{-\frac{2nt\sigma}{\sigma+1} + 2\lambda t\Big\}\ddr t\Big)^{1/2} \Big(\|f_{\alpha}\|_{\infty} \int_0^{+\infty} e^{-2\lambda t} f_{\alpha}\left(\frac{t}{z}\right) \ddr t\Big)^{1/2} \\
&\quad= \frac{e^{\sigma/2} \sqrt{1 + 1/\sigma}}{\sqrt{2n-1}} e^{-n\sigma} \sqrt{z} \|f_{\alpha}\|_{\infty}^{1/2} \exp\Big\{-\frac12\Big(\frac{\sigma z}{\sigma+1}\Big)^{\alpha}\Big\},
\end{align*}
which entails the desired asymptotic behavior. Second, take into account the region $s \in (-1,-\sigma)$. Writing $h(s) \geq h'(-\sigma)(s+\sigma) + h(-\sigma)$ for every $s \in (-1,-\sigma)$, and still using the convexity of the function $h$, we can write the following
$$
\int_{-1}^{-\sigma} e^{-n h(s)} f_{\alpha}\left(\frac{s+1}{z}\right) \ddr s \leq (1-\sigma)^n \int_0^{1-\sigma} \exp\Big\{\frac{nt\sigma}{1-\sigma}\Big\} f_{\alpha}\left(\frac{t}{z}\right) \ddr t\ .
$$
Therefore, the quantity to bound is now equal to 
$$
\Delta_2(z) := \frac{1}{f_{\alpha}(1/z)} \cdot \sup_{n \in \N} n^{3/2} (1-\sigma)^n \int_0^{1-\sigma} \exp\Big\{\frac{nt\sigma}{1-\sigma}\Big\} f_{\alpha}\left(\frac{t}{z}\right) \ddr t\ .
$$
For small $z$, argue as above using that $f_{\alpha}(1/z) = O(z^{1+\alpha})$ and $f_{\alpha}(x) \leq C_{\alpha}x^{-(1+\alpha)}$. For large values of $z$, exploit that $\sup_{t \in [0, 1-\sigma]} \exp\big\{\frac{nt\sigma}{1-\sigma}\big\} = e^{n\sigma}$ and conclude by using that $\int_0^{1-\sigma} \!\! f_{\alpha}\left(\frac{t}{z}\right) \ddr t \sim z^{\frac{1-3\alpha/2}{1-\alpha}} \exp\{-C_{\alpha} z^{\alpha/(1-\alpha)}\}$, as a consequence of \eqref{asymptotic_stable}. Third, to study the integral in the region $(0,\sigma)$, consider the inversion of the analytic function $h(s)$. Since $t = h(s) = \sum_{k=2}^{\infty} \frac{(-s)^k}{k}$ for $s \in (0,\sigma)$, by means of Lagrange's inversion formula $s = \sum_{k=1}^{\infty} \alpha_k t^{k/2}$. The coefficients $\alpha_k$ are given by $\alpha_1 = \sqrt{2}$, $\alpha_2 = 2/3$ and
$$
\frac{k+2}{\sqrt{2}} \alpha_{k+1} = \alpha_k - \sum_{j=0}^{k-2} \frac{j+2}{2} \alpha_{j+2}\alpha_{k-j}
$$
for $k=2,3, \ldots$. See, e.g., Example 1 in Chapter 2 of Wong \cite{Won(01)}.
Thus, $h : (0,\sigma) \rightarrow (0,h(\sigma))$ is bijective, with inverse function given by $q(t) := \sum_{k=1}^{\infty} \alpha_k t^{k/2}$ for $t \in (0,h(\sigma))$. These facts guarantee the possibility to change the variable, to get
$$
\int_0^{\sigma} e^{-n h(s)} f_{\alpha}\left(\frac{s+1}{z}\right) \ddr s = \int_0^{h(\sigma)} e^{-nt} f_{\alpha}\left(\frac{q(t)+1}{z}\right) q'(t)\ddr t\ .
$$
At this stage, invoke the Taylor formula to show that, for all $s \in (0,\sigma)$ and $t \in (0,h(\sigma))$, there holds $|F(s;z)| \leq \frac{1}{2} \sup_{y \in (0,\sigma)}\Big| f''_{\alpha}\Big(\frac{y+1}{z}\Big)\Big| \Big(\frac{s}{z}\Big)^2$, $|Q(t)| \leq Ct^{3/2}$ and $|Q_1(t)| \leq Ct^{1/2}$ with some numerical constant $C$, where we defined 
\begin{eqnarray*}
F(s;z) &:=& f_{\alpha}\Big(\frac{s+1}{z}\Big) - \Big[f_{\alpha}\Big(\frac{1}{z}\Big) + f'_{\alpha}\Big(\frac{1}{z}\Big)\frac{s}{z} \Big], \\
Q(t) &:=& q(t) - [\alpha_1 t^{1/2} + \alpha_2 t], \\
Q_1(t) &:=& q'(t) - [\frac12 \alpha_1 t^{-1/2} + \alpha_2].
\end{eqnarray*}
Finally, the study in the region $(-\sigma, 0)$ starts from the inversion of the analytic functions $h(-s)$ for $s \in (0,\sigma)$. Deduce that 
$s = \sum_{k=1}^{\infty} (-1)^{k+1}\alpha_k t^{k/2} =: \overline{q}(t)$ is the inverse of $t = h(-s)$ for $s \in (0,\sigma)$. Changing the variable yields
$$
\int_{-\sigma}^0 e^{-n h(s)} f_{\alpha}\left(\frac{s+1}{z}\right) \ddr s = \int_0^{h(-\sigma)} e^{-nt} f_{\alpha}\left(\frac{-\overline{q}(t)+1}{z}\right) \overline{q}'(t)\ddr t\ .
$$
Again, the Taylor formula shows that, for all $s \in (0,\sigma)$ and $t \in (0,h(-\sigma))$, there holds $|\overline{F}(s;z)|  \leq \frac{1}{2} \sup_{y \in (0,\sigma)}\Big| f''_{\alpha}\Big(\frac{-y+1}{z}\Big)\Big| \Big(\frac{s}{z}\Big)^2$, $|\overline{Q}(t)| \leq Ct^{3/2}$ and $|\overline{Q}_1(t)| \leq Ct^{1/2}$  with some numerical constant $C$, where we defined 
\begin{eqnarray*}
\overline{F}(s;z) &:=& f_{\alpha}\Big(\frac{-s+1}{z}\Big) - \Big[f_{\alpha}\Big(\frac{1}{z}\Big) - f'_{\alpha}\Big(\frac{1}{z}\Big)\frac{s}{z} \Big], \\
\overline{Q}(t) &:=& \overline{q}(t) - [\alpha_1 t^{1/2} - \alpha_2 t], \\
\overline{Q}_1(t) &:=& q'(t) - [\frac12 \alpha_1 t^{-1/2} - \alpha_2].
\end{eqnarray*}
Then, the way is paved to study
$$
\Delta_3(z) :=  \sup_{n \in \N}  n \Big| 
\frac{\int_0^{\sigma} e^{-n h(s)} f_{\alpha}\left(\frac{s+1}{z}\right) \ddr s + \int_{-\sigma}^0 e^{-n h(s)} f_{\alpha}\left(\frac{s+1}{z}\right) \ddr s}{f_{\alpha}(1/z)\sqrt{\frac{2\pi}{n}}} -1 \Big| 
$$
 by expanding the sum of the two integrals in the numerator. The multitude of the ensuing terms is then split into four groups, according to
the dependence on the $z$-variable. The first group corresponds to
\begin{align*}
f_{\alpha}\left(\frac{1}{z}\right) &\Big\{ \frac{\alpha_1}{2} \int_0^{h(\sigma)} e^{-nt} t^{-1/2} \ddr t + \alpha_2 \int_0^{h(\sigma)} e^{-nt} \ddr t + \int_0^{h(\sigma)} e^{-nt} Q_1(t)\ddr t \\
& + \frac{\alpha_1}{2} \int_0^{h(-\sigma)} \!\! e^{-nt} t^{-1/2} \ddr t - \alpha_2 \int_0^{h(-\sigma)} \!\! e^{-nt} \ddr t + \int_0^{h(-\sigma)} \!\!e^{-nt} \overline{Q}_1(t)\ddr t \Big\}\ .
\end{align*}
The key remark rests on the identity $\alpha_1 \int_0^{+\infty} e^{-nt} t^{-1/2} \ddr t = \sqrt{\frac{2\pi}{n}}$, so that the term $-1$ that appears in the definition of $\Delta_3(z)$ cancels out.
The other noteworthy simplification is obtained by considering the identity $\int_0^{h(\sigma)} e^{-nt} \ddr t - \int_0^{h(-\sigma)} e^{-nt} \ddr t = \frac1n [e^{-nh(-\sigma)} - e^{-nh(\sigma)}]$. The remaining terms are handled by means of the Watson lemma (Section 5.I of Wong \cite{Won(01)}). Therefore, this first group of terms contributes as a bounded function of
$z$, because of the simplification of the pre-factor $f_{\alpha}(1/z)$ with the same term appearing in the denominator of the expression that defines $\Delta_3$. The second group of terms is given by 
\begin{align*}
\frac{1}{z} f'_{\alpha}\left(\frac{1}{z}\right) &\Big\{ \frac{\alpha_1^2}{2} \int_0^{h(\sigma)}   \!\!\!\! e^{-nt} \ddr t + \alpha_1 \int_0^{h(\sigma)} \!\!\!\! e^{-nt} t^{1/2} [\alpha_2 + Q_1(t)]\ddr t \\
& +  \int_0^{h(\sigma)} \!\!\!\! e^{-nt} [\alpha_2 t + Q(t)] \cdot [\frac12 \alpha_1 t^{-1/2} + \alpha_2 + Q_1(t)] \ddr t \\
& - \frac{\alpha_1^2}{2} \int_0^{h(-\sigma)}   \!\!\!\! e^{-nt} \ddr t - \alpha_1 \int_0^{h(-\sigma)} \!\!\!\! e^{-nt} t^{1/2} [-\alpha_2 + \overline{Q}_1(t)]\ddr t \\
&+ \int_0^{h(-\sigma)} \!\!\!\! e^{-nt} [-\alpha_2 t + \overline{Q}(t)] \cdot [\frac12 \alpha_1 t^{-1/2} - \alpha_2 + \overline{Q}_1(t)] \ddr t \Big\}\ .
\end{align*}
As before, it is worth remarking the simplification involving the quantity $\int_0^{h(\sigma)} e^{-nt} \ddr t - \int_0^{h(-\sigma)} e^{-nt} \ddr t = \frac1n [e^{-nh(-\sigma)} - e^{-nh(\sigma)}]$. Again, the remaining terms are handled by means of the Watson lemma.  As to the asymptotic behavior, recall \eqref{humbert} and 
\eqref{asymptotic_stable} to get $f'_{\alpha}(z^{-1})/z f_{\alpha}(z^{-1}) = O(1)$ and $f'_{\alpha}(z^{-1})/z f_{\alpha}(z^{-1}) = O(z^{\frac{\alpha}{1-\alpha}})$
as $z \rightarrow 0$ and $z \rightarrow +\infty$, respectively. Finally, consider
$$
\frac{1}{2z^2} \sup_{y \in (0,\sigma)}\Big| f''_{\alpha}\Big(\frac{y+1}{z}\Big)\Big| \int_0^{h(\sigma)} e^{-nt} [q(t)]^2 \cdot [\frac12 \alpha_1 t^{-1/2} + \alpha_2 + Q_1(t)] \ddr t 
$$
and 
$$
\frac{1}{2z^2} \sup_{y \in (0,\sigma)}\Big| f''_{\alpha}\Big(\frac{-y+1}{z}\Big)\Big| \int_0^{h(-\sigma)} e^{-nt} [\overline{q}(t)]^2 \cdot [\frac12 \alpha_1 t^{-1/2} - \alpha_2 + \overline{Q}_1(t)] 
\ddr t\ . 
$$
Thanks to equations \eqref{humbert} and \eqref{asymptotic_stable}, the asymptotic behavior of these expressions stems from $f''_{\alpha}(z^{-1})/z^2 f_{\alpha}(z^{-1}) = O(1)$ and $f''_{\alpha}(z^{-1})z^2 f_{\alpha}(z^{-1}) = O(z^{\frac{2\alpha}{1-\alpha}})$ as $z \rightarrow 0$ and $z \rightarrow +\infty$, respectively. This completes the proof.
\end{proof}

\subsection{A quantitative Poisson approximation} \label{sect:poisson}

Our result improves Theorem 1 of Hwang \cite{Hwa(99)} by reformulating it as a true inequality, i.e., 
without ``big O" terms.  As to notation, for any $r > 0$ and $z_0 \in \C$, $D_r(z_0)$ ($\overline{D_r}(z_0)$, respectively) will denote the open (closed, respectively) disc in $\C$ of radius $r$, centered at $z_0$.

\begin{prp} \label{prop:hwang}
Let $\ggsf_{X_n}$ be holomorphic in $D_{\eta+\tau_n}(0)$ for every $n \in \N$, with $\eta > 3$ independent of $n$ and $\tau_n > 0$. Moreover, suppose that $\ggsf_{X_n}$ can be written as
\begin{equation} \label{hwang_GF}
\ggsf_{X_n}(s) = \exp\{\lambda_n (s-1)\} s^h [g(s) + \epsilon_n(s)]
\end{equation}
for every $s \in D_{\eta+\tau_n}(0)$, where:
\begin{enumerate}
\item[i)] $\{\lambda_n\}_{n\geq 1}$ is a diverging sequence of positive numbers;
\item[ii)] $h \in \N_0$ is a constant, independent of $n$; 
\item[iii)] the restriction of $g$ to $D_{\eta}(0)$ is independent of $n$ and holomorphic on that domain, with $g(1)=1$ and $g(0)\neq 0$;
\item[iv)] $\epsilon_n$ is holomorphic in $D_{\eta}(0)$ and satisfies $K(\eta-1) < +\infty$, where 
\end{enumerate}
$$
K(\delta) := \sup_{n \in \N}\ \ \sup_{0 < |s-1| < \delta} \lambda_n\Big| \frac{\epsilon_n(s)}{s-1}\Big| < +\infty \ .
$$
Then, there exist some $n_0 \in \N$ and $C(\eta) > 0$, independent of $n$, such that 
\begin{equation} \label{hwang_TESI}
\sum_{k \geq h} \Big| \ppsf[X_n = k] - \exp\{-[\lambda_n + g'(1)]\} \frac{[\lambda_n + g'(1)]^{k-h}}{(k-h)!} \Big| \leq \frac{C(\eta)}{\lambda_n}
\end{equation}
holds for every $n \geq n_0$.
\end{prp}

The following proof contains also a quantification of $C(\eta)$. In addition, notice that $\lambda_n + g'(1) > 0$ holds eventually (with respect to $n$), and that the integer $h$ is well-defined in view of $g(0)\neq0$ and the holomorphic character of $g$ and $\epsilon_n$ about $s = 0$.

\begin{proof}[Proof of Proposition \ref{prop:hwang}]
Set $\sigma_n := \lambda_n + g'(1)$. Write $\ppsf[X_n = k+h] = I_1^{(k,n)}(\rho) + I_2^{(k,n)}(\rho)$, where, in view of \eqref{hwang_GF} and the Cauchy formula, 
\begin{eqnarray*}
I_1^{(k,n)}(\rho) &=& \frac{1}{2\pi i} \oint_{D_{\rho}(0)} g(s) s^{-(k+1)} e^{\lambda_n(s-1)} \ddr s\\
I_2^{(k,n)}(\rho) &=& \frac{1}{2\pi i} \oint_{D_{\rho}(0)} \epsilon_n(s) s^{-(k+1)} e^{\lambda_n(s-1)} \ddr s
\end{eqnarray*}
for any $\rho \in (0, \eta)$. Observe that the left-hand side of \eqref{hwang_TESI} is bounded by
\begin{equation} \label{hwang_sommatotale}
\sum_{k \geq 0} \left[\big| I_1^{(k,n)}(\rho) - e^{-\sigma_n} \frac{\sigma_n^k}{k!} \big|\ +\ \big| I_2^{(k,n)}(\rho) \big|\right]\ .
\end{equation}
Fix $\delta \in (1, \eta-1)$ and define $M_1(n) := \lambda_n - \lambda_n^{1/2+1/7}$ and $M_2(n) := \lambda_n + \lambda_n^{1/2+1/7}$, yielding $M_1(n) < M_2(n) < \delta\lambda_n$ definitely with respect to $n$. Then, \eqref{hwang_sommatotale} is majorized by
\begin{align} 
& \sum_{0 \leq k < M_1(n) } \left[e^{-\sigma_n} \frac{\sigma_n^k}{k!}\ +\ \big| I_1^{(k,n)}(\rho) \big|\ +\ \big| I_2^{(k,n)}(\rho) \big|\right] \nonumber \\
&\quad+ \sum_{M_1(n) \leq k < M_2(n)} \left[\big| I_1^{(k,n)}(\rho) - e^{-\sigma_n} \frac{\sigma_n^k}{k!} \big|\ +\ \big| I_2^{(k,n)}(\rho) \big|\right] \nonumber \\
&\quad+ \sum_{k \geq M_2(n)} e^{-\sigma_n} \frac{\sigma_n^k}{k!} + \sum_{M_2(n) \leq k <  \delta\lambda_n} \left[\big| I_1^{(k,n)}(\rho) \big|\ +\ \big| I_2^{(k,n)}(\rho) \big|\right] \nonumber \\
&\quad+ \sum_{k \geq \delta\lambda_n} \left[\big| I_1^{(k,n)}(\rho) \big|\ +\ \big| I_2^{(k,n)}(\rho) \big|\right]\ .  \label{hwang_sommaspezzata}
\end{align} 
For the last sum in \eqref{hwang_sommaspezzata}, choose $\rho = \delta$ to obtain
$$
\big| I_1^{(k,n)}(\delta) \big|\ +\ \big| I_2^{(k,n)}(\delta) \big| \leq \frac12\sqrt{\frac{\pi}{2}} \left[\eta K(\eta)\lambda_n^{-1} + \sup_{s \in D_{\eta}(0)} |g(s)|\right] \lambda_n^{-1/2} \delta^{-k} e^{\lambda_n(\delta-1)}
$$
for all $k \geq \delta\lambda_n$. Hence, noticing that $1- \delta + \delta\log\delta > 0$, conclude that
\begin{equation} \label{hwang_conclusion_1eps}
\sum_{k \geq \delta\lambda_n} \left[\big| I_1^{(k,n)}(\delta) \big|\ +\ \big| I_2^{(k,n)}(\delta) \big|\right] \leq C_1(\eta, \delta, g) \lambda_n^{-1/2} 
e^{-\lambda_n(1- \delta + \delta\log\delta)}\ .
\end{equation}
Then, if $k \in (0, M_1(n)) \cup [M_2(n),\delta\lambda_n)$, set $\rho := k/\lambda_n$ to get
\begin{eqnarray*}
\big| I_1^{(k,n)}(k/\lambda_n) \big| &\leq& \frac{e\sqrt{\pi}}{2\sqrt{2}} \big[\sup_{s \in D_{\eta}(0)} |g(s)|\big] e^{-\lambda_n} \frac{\lambda_n^k}{k!}\\
\big| I_2^{(k,n)}(k/\lambda_n) \big| &\leq& \frac{e\sqrt{\pi}}{2\sqrt{2}} K(\eta) \left(\frac{1}{\lambda_n}\right) 
\left[ \Big|\frac{k}{\lambda_n} -1\Big| + \sqrt{\frac{\pi\delta}{2\lambda_n}}\right] e^{-\lambda_n} \frac{\lambda_n^k}{k!}\ .
\end{eqnarray*}
By Chernoff bounds for Poisson probabilities, for the fourth sum in \eqref{hwang_sommaspezzata} there holds
\begin{displaymath}
\sum_{M_2(n) \leq k <  \delta\lambda_n} \left[\big| I_1^{(k,n)}(\rho) \big|\ +\ \big| I_2^{(k,n)}(\rho) \big|\right]  \leq C_2(\eta, \delta, g) \exp\left\{-\frac12 \lambda_n^{2/7}\right\}
\end{displaymath}
while, for the third sum in \eqref{hwang_sommaspezzata}, there exists a constant $C_3$ such that
$$
\sum_{0 \leq k < M_1(n)}e^{-\sigma_n} \frac{\sigma_n^k}{k!} + \sum_{k \geq M_2(n)}e^{-\sigma_n} \frac{\sigma_n^k}{k!} \leq C_3 \exp\left\{-\frac12 \lambda_n^{2/7}\right\}\ .
$$
Furthermore, the same argument based on the Chernoff bounds for Poisson probabilities shows that
\begin{displaymath}
\sum_{0 \leq k <  M_1(n)} \left[\big| I_1^{(k,n)}(\rho) \big|\ +\ \big| I_2^{(k,n)}(\rho) \big|\right]  \leq C_4(\eta, \delta, g) \exp\left\{-\frac12 \lambda_n^{2/7}\right\}\ .
\end{displaymath}
It remains the second sum in \eqref{hwang_sommaspezzata}, which brings the main contribution. In particular, for the sum relative to $\epsilon_n$,
choose $\rho := k/\lambda_n$ and  write
\begin{align*} 
\big| I_2^{(k,n)}(k/\lambda_n) \big| &\leq \left(e^{-\lambda_n} \frac{\lambda_n^k}{k!}\right) \left(\frac{k! e^k}{2\pi k^k\lambda_n}\right)  \nonumber \\
&\quad\quad\times \Big[ K(\delta_n) \int_{|\theta| \leq \lambda_n^{-5/14}} \big| \frac{k}{\lambda_n} e^{i\theta} - 1\big{|} e^{-k(1 - \cos\theta)} \ddr \theta \nonumber \\
&\quad+ K(\eta) \int_{|\theta| > \lambda_n^{-5/14}} \big| \frac{k}{\lambda_n} e^{i\theta} - 1\big{|} e^{-k(1 - \cos\theta)} \ddr \theta\Big] \nonumber \\
&\leq C_5(\eta) \left(e^{-\lambda_n} \frac{\lambda_n^k}{k!}\right) \Big[ K(\delta_n) \exp\left\{-\frac15 \lambda_n^{2/7}\right\} + K(\delta_n) \lambda_n^{-19/14}\Big] \nonumber
 \end{align*}
where $\delta_n := \sqrt{[1 - (1- \lambda_n^{-5/14})\cos\lambda_n^{-5/14}]^2 + [(1- \lambda_n^{-5/14})\sin\lambda_n^{-5/14}]^2} \sim \sqrt{2}\lambda_n^{-5/14}$. Whence,
\begin{displaymath}
\sum_{M_1(n) \leq k <  M_2(n)} \big| I_2^{(k,n)}(\rho) \big| \leq C_5(\eta) \Big[ K(\delta_n) \exp\left\{-\frac15 \lambda_n^{2/7}\right\} + K(\delta_n) \lambda_n^{-19/14}\Big] \ .
\end{displaymath}
Finally, upon noticing that
$$
I_1^{(k,n)}(\rho) - e^{-\sigma_n} \frac{\sigma_n^k}{k!} = \frac{1}{2\pi i} \oint_{D_{\rho}(0)} [g(s) - e^{g'(1)(s-1)}] s^{-(k+1)} e^{\lambda_n(s-1)} \ddr s
$$
is valid, write the integral on the right-hand side as
\begin{align*}
& \frac{g''(1)-[g'(1)]^2}{4\pi i} \oint_{D_{\rho}(0)} (s-1)^2 s^{-(k+1)} e^{\lambda_n(s-1)} \ddr s \\
&+ \sum_{m=3}^{\infty} \frac{1}{m!} \{g^{(m)}(1) - [g'(1)]^m\}  \frac{1}{2\pi i} \oint_{D_{\rho}(0)}(s-1)^m s^{-(k+1)} e^{\lambda_n(s-1)} \ddr s\ .
\end{align*}
Choosing $\rho := k/\lambda_n$ once again, and taking account of the well-known Cauchy estimates for holomorphic functions, conclude that the modulus of the above sum is bounded by
$$
e^{-\lambda_n} \frac{\lambda_n^k}{k!} \Big[ \frac12 |g''(1)-[g'(1)]^2| \cdot |\mathcal{C}_2(\lambda_n,k)| 
+ \left(\frac{1}{\lambda_n}\right)^{\frac{15}{14}} \sum_{m=3}^{\infty} 2^{m+1} \lambda_n^{-(m-3)/14}M_n(g)^m\Big]
$$
where $\mathcal{C}_2(\lambda_n,k) := \frac{k^2 - (2\lambda_n+1)k + \lambda_n^2}{\lambda_n^2}$ stands for the Poisson-Charlier polynomial of degree 2 and $M_n(g) := \sup_{|s-1| \leq \lambda_n^{-2/7}} |g(s)|$.
To conclude the proof, recall that $\sum_{k=0}^{\infty} e^{-\lambda_n} \frac{\lambda_n^k}{k!}\big| \mathcal{C}_2(\lambda_n,k) \big| \leq C_6 \lambda_n^{-1}$ holds with some numerical constant $C_6$ by virtue of  
Proposition 1 of Hwang \cite{Hwa(99)}, and observe that $2 \lambda_n^{-1/14} M_n(g) \leq \frac12$ for all $n$ sufficiently large. Whence, for all $n \geq n_0$, write
\begin{equation} \label{hwang_conclusion_5eps}
\sum_{M_1(n) \leq k <  M_2(n)} \big| I_1^{(k,n)}(\rho) - e^{-\sigma_n} \frac{\sigma_n^k}{k!} \big| \leq \frac{C_7}{\lambda_n}\Big[\frac{1}{2} \big|g''(1)-[g'(1)]^2\big| + 1\Big]
\end{equation}
for some suitable numerical constant $C_7$. This fact completes the proof.
\end{proof}

A sharper bound than the one obtained in Proposition \ref{prop:hwang} can be obtained by strengthening the hypotheses on $g$. This is stated in the following

\begin{cor} \label{cor:hwang2}
In addition to the assumptions made in Proposition \ref{prop:hwang}, suppose that   
\begin{equation} \label{hwang_g}
\Big| \sum_{l=0}^r \binom{r}{l} g^{(l)}(1) [-g'(1)]^{r-l} \Big| \leq \sqrt{r!} B^r
\end{equation}
holds for every $r \in \N_0$ and some $B>0$ independent of $r$. Then, it holds
\begin{displaymath}
\sum_{k \geq 0} \Big| I_1^{(k,n)}(\rho) - \exp\{-[\lambda_n + g'(1)]\} \frac{[\lambda_n + g'(1)]^{k}}{k!} \Big| \leq \frac{C_{\ast}(\eta)}{\lambda_n}
\end{displaymath}
for every $n \geq n_0$, $C_{\ast}(\eta)$ being a suitable constant depending solely on $\eta$.
\end{cor}

\begin{proof}
Rewrite $I_1^{(k,n)}(\rho)$ by means of an identity due to Ch. Jordan (see Kullback \cite{Kull(47)} and Uspensky \cite{Usp(31)}). For any $x > 0$ and 
$r \in \N_0$, define
$$
L_r(x; n) := \frac{r!}{2\pi i} \oint_{D_{\overline{\rho}}(0)} g(1+s) s^{-(k+1)} e^{(\lambda_n - x)s} \ddr s = \sum_{l=0}^r \binom{r}{l} g^{(l)}(1) (\lambda_n - x)^{r-l}
$$
the radius $\overline{\rho}$ being any number in $(1, \eta-1)$. See formulae (6) and (2.3) in Uspensky \cite{Usp(31)} and Kullback \cite{Kull(47)}. Setting $x = \sigma_n := \lambda_n + g'(1)$, Jordan's identity reads
\begin{equation} \label{Jordan}
I_1^{(k,n)}(\rho) = \sum_{r=0}^{\infty} \frac{1}{r!} L_r(\sigma_n; n) \mathcal{C}_r(\sigma_n; k) e^{-\sigma_n} \frac{\sigma_n^k}{k!} 
\end{equation}
the symbol $\mathcal{C}_r(x; k) := \sum_{l=0}^r \binom{r}{l} (-1)^{r-l} [k]_{[l]} x^{-l}$ standing for the Poisson-Charlier polynomial, where $[k]_{[j]}$ denotes the falling factorial, i.e. 
$[k]_{[0]} := 1$ and $[k]_{[j]} := \prod_{m=0}^{j-1} (k-m)$ if $j \in \N$. At this stage, the argument used to prove Proposition 1 of Hwang \cite{Hwa(99)} (see also Lemma 6 in Shorgin \cite{Sho(78)}) shows that, for any $x_0 > 1$, there exists a constant $M(x_0)$ depending solely by $x_0$ such that
\begin{equation} \label{hwang_prop_1}
\sum_{k=0}^{\infty} e^{-x}\frac{x^k}{k!} |\mathcal{C}_r(x; k)| \leq [M(x_0)]^r \Gamma\left(\frac{r+1}{2}\right) x^{-r/2}
\end{equation}
holds for every $x \geq x_0$ and $r \in \N_0$. Combination of \eqref{hwang_g}--\eqref{hwang_prop_1} yields
\begin{equation} \label{hwang_conclusion_1}
\sum_{k \geq 0} \Big|I_1^{(k,n)}(\rho)  - e^{-\sigma_n} \frac{\sigma_n^k}{k!} \Big| \leq \frac{B M(x_0)}{\sigma_n} \sum_{r=2}^{\infty} \Gamma\left(\frac{r+1}{2}\right) \frac{1}{\sqrt{r!}} \left(\frac{B M(x_0)}{\sigma_n}\right)^{(r-2)/2} 
\end{equation}
whenever $\sigma_n \geq x_0 > 1$, since $L_0(\sigma_n; n) = 1$ and $L_1(\sigma_n; n) = 0$. Thus, if $\sigma_n \geq 2BM(x_0)$, the series on the right-hand side of \eqref{hwang_conclusion_1} is bounded by a numerical constant, say $c_1$. Since $-2g'(1) \leq \lambda_n$ holds eventually in $n$, entailing that $1/\sigma_n \leq 2/\lambda_n$, the 
LHS of \eqref{hwang_conclusion_1} is bounded by $2c_1B M(x_0)/\lambda_n$, for all $n$ greater or equal than some $n_0$. This completes the proof.
\end{proof}


Our application of Proposition \ref{prop:hwang} starts from the evaluation of the probability generating function of the random variable 
$R(\alpha, n, t n^{\alpha})$, which is playing the role of $X_n$. In fact, a combination of equations \eqref{cm_dist} and \eqref{InBernardo} with Lemma \ref{lm:Laplace1} yields
\begin{equation} \label{bernardo_hwuang}
\ggsf_{R(\alpha, n, t n^{\alpha})}(s)
\ =\ e^{tn^{\alpha}(s-1)} s \frac{ \big( \frac{1}{s} \big)^{\frac{1}{\alpha} + 1} f_{\alpha}\left(\big(\frac{1}{st}\big)^{\frac{1}{\alpha}}\right) \Big[1 + \delta_n\big( (st)^{\frac{1}{\alpha}} \big)\Big]}{f_{\alpha}\left( \big(\frac{1}{t}\big)^{\frac{1}{\alpha}} \right) \Big[1 + \delta_n\big( t^{\frac{1}{\alpha}} \big) \Big]}\ .
\end{equation}
Currently, \eqref{bernardo_hwuang} holds for all $s,t > 0$, but it will be shown that the numerator on the right-hand side can be analytically continued, as a function of $s$, to the whole complex plane. To parallel \eqref{bernardo_hwuang} with \eqref{hwang_GF}, set $\lambda_n = tn^{\alpha}$, $h=1$, 
$$
g(s) = \frac{ \big( \frac{1}{s} \big)^{\frac{1}{\alpha} + 1} f_{\alpha}\left(\big(\frac{1}{st}\big)^{\frac{1}{\alpha}}\right)}{f_{\alpha}\left( \big(\frac{1}{t}\big)^{\frac{1}{\alpha}} \right)}\ \ \ \ \ \text{and}\ \ \ \ \ \epsilon_n(s) = g(s)
\left[ \frac{1 + \delta_n\big( (st)^{\frac{1}{\alpha}} \big)}{1 + \delta_n\big( t^{\frac{1}{\alpha}} \big)} - 1 \right]\ .
$$
As $\eta$, it can be chosen as any number strictly greater than 3, since $\ggsf_{R(\alpha, n, t n^{\alpha})}$ is an entire function. Apropos of the analytic continuation of the right-hand side of 
\eqref{bernardo_hwuang}, recall the definition of the Wright-Mainardi function $M_{\alpha}(z) := \frac{1}{\pi} \sum_{n=1}^{\infty} \frac{(-z)^{n-1}}{(n-1)!} \Gamma(\alpha n) \sin(\pi\alpha n)$, 
for $z \in \C$ and $\alpha \in (0,1)$, which coincides with $\frac{1}{\alpha z^{1 + 1/\alpha}} f_{\alpha}\big(\frac{1}{z^{1/\alpha}}\big)$ if $\Im(z) = 0$ and $\Re(z) > 0$ (Mainardi et al. \cite{Mai(10)}). It turns out that
$g$ can be re-written as $g(s) = M_{\alpha}(st)/M_{\alpha}(t)$, for all $s \in \C$, and \eqref{bernardo_hwuang} reads
$$
\ggsf_{R(\alpha, n, t n^{\alpha})}(s) = e^{tn^{\alpha}(s-1)} s \frac{ \int_{-1}^{+\infty} e^{nh(\sigma)} (1+\sigma)^{-\alpha} M_{\alpha}(st (1+\sigma)^{-\alpha}) \ddr \sigma}
{\int_{-1}^{+\infty} e^{nh(\sigma)} (1+\sigma)^{-\alpha} M_{\alpha}(t (1+\sigma)^{-\alpha}) \ddr \sigma} \ .
$$
Furthermore, the rate $\lambda_n + g'(1)$ of the shifted Poisson distribution in \eqref{hwang_TESI} becomes
$$
\omega(t; n,\alpha) := \lambda_n + g'(1) = tn^{\alpha} + \frac{t M'_{\alpha}(t)}{M_{\alpha}(t)}
$$
which is positive whenever $n$ is sufficiently large and $t \leq T(n, \alpha)$. The determination of the asymptotic behavior of $T(n,\alpha)$ is determined by the behavior of $M_{\alpha}$ for large arguments. In fact, as shown in the work of Mainardi et al. \cite{Mai(10)}, $M_{\alpha}(t/\alpha) \sim (2\pi(1-\alpha))^{-1/2} t^{\frac{\alpha - 1/2}{1 - \alpha}} \exp\left\{- \frac{1-\alpha}{\alpha} t^{\frac{1}{1 - \alpha}} \right\}$ as $t \rightarrow +\infty$, in agreement with \eqref{asymptotic_stable}.  

At this stage, this subsection is completed by stating a corollary that matches Proposition \ref{prop:hwang} with Lemma \ref{lm:Laplace1}. 
After recalling that the symbol $N_{\lambda}$ denotes a Poisson random variable with parameter $\lambda$, we have the following
\begin{cor}\label{coro:hwang} 
There exists $n_0$ independent of $t$ such that, for all $n \geq n_0$ and $t \leq T(n,\alpha)$, there holds
\begin{displaymath}
\sum_{k=1}^{\infty} \big| \ppsf[R(\alpha, n, t n^{\alpha}) = k] - \ppsf[1 + N_{\omega(t; n,\alpha)} = k] \big| \leq \frac{\Upsilon(t)}{t n^{\alpha}}
\end{displaymath}
where $T(n,\alpha) \sim n^{\kappa\alpha}$ as $n\rightarrow +\infty$ for some $\kappa \in (0,1)$, and $\Upsilon: (0,+\infty) \rightarrow (0,+\infty)$ is a suitable continuous function which is independent of $n$. Moreover, 
$\Upsilon$ can be chosen in such a way that $\Upsilon(t) = O(1)$ as $t \rightarrow 0$, and $\Upsilon(t)\cdot f_{\alpha}\left(\frac{1}{t^{1/\alpha}}\right) = O(t^{-\infty})$ as $t \rightarrow +\infty$.
\end{cor}

\begin{proof}
It is enough to recover the bounds \eqref{hwang_conclusion_1eps}-\eqref{hwang_conclusion_5eps}. After they are rewritten in terms of $t$, it is enough to exploit the properties of the function $\Delta$ in 
Lemma \ref{lm:Laplace1}.
\end{proof}


\subsection{Conclusion} \label{sect:conclusion}

This subsection contains the heart of the proof of Theorem \ref{be}, whose strategy consists in four main steps. Maintaining the notation used in Subsections \ref{sect:representation} and \ref{sect:poisson}, these steps can be summarized as follows:
\begin{enumerate} 
\item[A)] since the strong law of large numbers hints that $G_{\theta+n,1}^{\alpha} \sim n^{\alpha}$ as $n \rightarrow +\infty$, with probability 1, one tries to rigorously prove that the probability laws of 
$K_n$ and $R(\alpha, n, S_{\alpha,\theta} \cdot n^{\alpha})$ are close, even in total variation;
\item[B)] after conditioning on the hypothesis that $S_{\alpha,\theta} = t$, one invokes Corollary \ref{coro:hwang} to prove that the probability laws of $R(\alpha, n, t n^{\alpha})$ and $1+N_{\omega(t, n,{\alpha})}$
are close in total variation;
\item[C)] by resorting to well-known results about Poisson mixtures contained in \cite{LeCam(55)} (see also Proposition \ref{prop:LeCam} below), one shows that the probability laws of 
$N_{\omega(S_{\alpha,\theta}, n,{\alpha})}$ and $N_{S_{\alpha, \theta} \cdot n^{\alpha}}$ are close in total variation, the random variable $S_{\alpha,\theta}$ being thought of as independent of the family of random variables $\{N_{\lambda}\}_{\lambda > 0}$;
\item[D)] in view of a quantitative law of large numbers for the Poisson process stated in Adell and de la Cal \cite{AdCal(93)}, one concludes by checking that 
the probability laws of $(1+N_{S_{\alpha, \theta} \cdot n^{\alpha}})/n^{\alpha}$ and $S_{\alpha, \theta}$ are close in the Kolmogorov metric. 
\end{enumerate} 

To grasp our strategy, write $\ddr_K(X; Y)$ to denote the Kolmogorov distance between the distribution functions of $X$ and $Y$, and define $\Omega_{n,\alpha,\theta} := 
\omega(S_{\alpha,\theta}; n, \alpha) \ind\{S_{\alpha,\theta} \leq T(n,\alpha)\} \geq 0$, $T(n,\alpha)$ being the same as in Corollary \ref{coro:hwang}, to get
\begin{align*}
\ddr_K\big(K_n/n^{\alpha}; S_{\alpha,\theta}\big) & \leq \ddr_K\big(K_n/n^{\alpha}; R(\alpha, n,  n^{\alpha}S_{\alpha,\theta})/n^{\alpha}\big) \\
&\quad + \ddr_K\big(R(\alpha, n, n^{\alpha}S_{\alpha,\theta})/n^{\alpha}; (1+N_{\Omega_{n,\alpha, \theta}})/n^{\alpha}\big) \\
&\quad + \ddr_K\big((1+N_{\Omega_{n,\alpha, \theta}})/n^{\alpha}; (1+N_{S_{\alpha, \theta} n^{\alpha}})/n^{\alpha}\big) \\
&\quad + \ddr_K\big((1+N_{S_{\alpha, \theta}n^{\alpha}})/n^{\alpha}; S_{\alpha, \theta}\big) 
\end{align*}
where $S_{\alpha,\theta}$ is independent of both $\{R(\alpha, n, z)\}_{z>0}$ and $\{N_{\lambda}\}_{\lambda > 0}$. Moreover, $N_0$ is intended, henceforth, 
as the degenerate random variable equal a.s. to 0. After recalling that 
$\ddr_K(aX+b; aY+b) = \ddr_K(X; Y)$ for any $a > 0$ and $b \in \R$, and that $\ddr_K(X; Y) \leq \ddr_{TV}(X; Y) := \sup_{B \in \mathscr{B}(\R)} |\ppsf[X\in B] - \ppsf[Y\in B]|$, $\ddr_{TV}$ being the total variation distance, one can write
\begin{align}
\ddr_K\big(K_n/n^{\alpha}; S_{\alpha,\theta}\big) &\leq \ddr_{TV}\big(K_n; R(\alpha, n, n^{\alpha}S_{\alpha,\theta})\big) \label{generalsplit} \\
&\quad+ \int_0^{T(n,\alpha)} \!\!\!\!\! \ddr_{TV}\big(R(\alpha, n, t n^{\alpha}); 1+N_{\omega(t;n,\alpha, \theta)}\big) f_{S_{\alpha,\theta}}(t) \ddr t \nonumber  \\
&\quad+ \ppsf\big[S_{\alpha, \theta} > T(n,\alpha) \nonumber \big] \\
&\quad+ \ddr_K\big(N_{\Omega_{n,\alpha, \theta}}; N_{S_{\alpha, \theta} \cdot n^{\alpha}}\big) + \ddr_K\big((1+N_{S_{\alpha, \theta} \cdot n^{\alpha}})/n^{\alpha}; S_{\alpha, \theta}\big)\ . \nonumber
\end{align}

Hence, one starts by studying the first term on the RHS of \eqref{generalsplit}.
\begin{prp} \label{prop:first_interpolant}
For fixed $\alpha \in (0,1)$ and $\theta > 0$, there exists a positive constant $C_1(\alpha, \theta)$ such that
\begin{align} \label{eq:mainstep1a}
& 2\ddr_{TV}\big(K_n; R(\alpha, n, n^{\alpha}S_{\alpha,\theta})\big)\\
&\notag\quad= \sum_{k=1}^n \Big| \ppsf\big[K_n = k \big] - \ppsf\big[R(\alpha, n, n^{\alpha}S_{\alpha,\theta}) = k \big]\Big| \leq C_1(\alpha, \theta)/n
\end{align}
holds for all $n \in \N$.
\end{prp}

\begin{proof}[Proof of Proposition \ref{prop:first_interpolant}]
The law of $R(\alpha, n, S_{\alpha,\theta} \cdot n^{\alpha})$ is given by
\begin{equation}\label{eq:first_interpolant}
\ppsf[R(\alpha, n, S_{\alpha,\theta} \cdot n^{\alpha}) = k] = \int_0^{+\infty} \left[\frac{\Ccr(n,k;\alpha) (t n^{\alpha})^k}{\sum_{j=1}^n \Ccr(n,j;\alpha) (t n^{\alpha})^j}\right] f_{S_{\alpha,\theta}}(t) \ddr t
\end{equation}
for $k = 1, \dots, n$. Combining \eqref{dist_prior}, \eqref{InBernardo} and \eqref{eq:first_interpolant}, the LHS of \eqref{eq:mainstep1a} is equal to
$$
\sum_{k=1}^n \Big| \Ccr(n,k;\alpha) \frac{\Gamma(k + \theta/\alpha)}{\Gamma(\theta/\alpha)} \frac{\Gamma(\theta)}{\Gamma(n + \theta)} - \int_0^{+\infty} \frac{\Ccr(n,k;\alpha) (tn^{\alpha})^k}{d_n(t)} f_{S_{\alpha,\theta}}(t) \ddr t\Big|\ .
$$
Set $d_n^{\ast}(t) := e^{tn^{\alpha}} (n-1)! \frac{1}{t^{1/\alpha}} f_{\alpha}\big(\frac{1}{t^{1/\alpha}}\big)$ and majorize the above quantity by
\begin{align*}
& \sum_{k=1}^n \Big| \Ccr(n,k;\alpha) \frac{\Gamma(k + \theta/\alpha)}{\Gamma(\theta/\alpha)} \frac{\Gamma(\theta)}{\Gamma(n + \theta)} - \int_0^{+\infty} \frac{\Ccr(n,k;\alpha) (tn^{\alpha})^k}{d_n^{\ast}(t)} f_{S_{\alpha,\theta}}(t) \ddr t\Big|\ \nonumber \\
&\quad+ \int_0^{+\infty} \frac{|d_n^{\ast}(t) - d_n(t)|}{d_n^{\ast}(t)} f_{S_{\alpha,\theta}}(t) \ddr t\ .
\end{align*}
Since $\int_0^{+\infty} \frac{(tn^{\alpha})^k}{d_n^{\ast}(t)} f_{S_{\alpha,\theta}}(t) \ddr t = \frac{1}{(n-1)!} \frac{\Gamma(k + \theta/\alpha)}{n^{\theta}} \frac{\Gamma(\theta)}{\Gamma(\theta/\alpha)}$ holds after taking account of \eqref{ml}, one has
\begin{align*}
& \sum_{k=1}^n \Big| \Ccr(n,k;\alpha) \frac{\Gamma(k + \theta/\alpha)}{\Gamma(\theta/\alpha)} \frac{\Gamma(\theta)}{\Gamma(n + \theta)} - \int_0^{+\infty} \frac{\Ccr(n,k;\alpha) (tn^{\alpha})^k}{d_n^{\ast}(t)} f_{S_{\alpha,\theta}}(t) \ddr t\Big| \\
&\quad = \Big| 1 - \frac{\Gamma(\theta+n)}{\Gamma(n) n^{\theta}}\Big| \leq \frac{([\theta]+1)! - 1}{n}
\end{align*}
where: i) the identity is a consequence of \eqref{dist_prior}; ii) the inequality, in which $[\theta]$ denotes the integral part of $\theta$, follows from the well-known Tricomi-Erdelyi expansion of the gamma ratio. 
Finally, by resorting to \eqref{InBernardo}--\eqref{Laplace1}, one has
\begin{align*}
&\int_0^{+\infty} \frac{|d_n^{\ast}(t) - d_n(t)|}{d_n^{\ast}(t)} f_{S_{\alpha,\theta}}(t) \ddr t \\
&\quad\leq \Big| \frac{(n/e)^n \sqrt{2\pi n}}{n!} - 1 \Big| + \left(\frac{(n/e)^n \sqrt{2\pi n}}{n!}\right) \frac{1}{n} \int_0^{+\infty} \Delta(t^{1/\alpha}) f_{S_{\alpha,\theta}}(t) \ddr t
\end{align*}
which leads to the desired conclusion, in view of the well-known Stirling approximation and the fact that $\int_0^{+\infty} \Delta(t^{1/\alpha}) f_{S_{\alpha,\theta}}(t) \ddr t < +\infty$, by virtue of Lemma \ref{lm:Laplace1}. Indeed, suffice it to recall that $\Delta(t^{1/\alpha}) = O(1)$ and $f_{S_{\alpha,\theta}}(t) \sim t^{\theta/\alpha}$ as $t \rightarrow 0$, whereas the condition
$\Delta(z) f_{\alpha}(1/z) = O(z^{-\infty})$, valid as $z \rightarrow +\infty$, entails the integrability of $\Delta(t^{1/\alpha}) f_{S_{\alpha,\theta}}(t)$ at infinity.
\end{proof}

The conclusion of the proof of Theorem \ref{be} requires to combine results of Subsection \ref{sect:poisson} with other known results. According to point B), invoke 
Corollary \ref{coro:hwang} to obtain
\begin{align*}
& \int_0^{T(n,\alpha)} \left(\sum_{k=1}^{\infty} \Big| \ppsf[R(\alpha, n, t n^{\alpha}) = k] - \ppsf[1 + N_{\omega(t; n,\alpha)} = k] \Big|\right) f_{S_{\alpha,\theta}}(t) \ddr t \\
&\quad \leq C_2(\alpha, \theta)/n^{\alpha} \nonumber
\end{align*}
where $C_2(\alpha, \theta) := \int_0^{+\infty} t^{-1}\Upsilon(t) f_{S_{\alpha,\theta}}(t)  \ddr t < +\infty$ in view of the asymptotic properties of $f_{S_{\alpha,\theta}}$. In order to proceed with point C), notice that
\begin{displaymath}
\ppsf[\omega(S_{\alpha,\theta}; n, \alpha) < 0] = \int_{T(n,\alpha)}^{+\infty} f_{S_{\alpha,\theta}}(t) \ddr t  \sim \frac{1}{n^{\alpha}}
\end{displaymath}
as $n \rightarrow \infty$. 

To bound the fourth term on the right-hand side of \eqref{generalsplit}, note that this is the Kolmogorov distance between two Poisson mixtures. This problem is tackled in Theorem 1 of Le Cam \cite{LeCam(55)}, which is re-stated for the reader's ease.
\begin{prp}[Le Cam] \label{prop:LeCam}
Given two probability measures $\gamma_1, \gamma_2$ on $\mathscr{B}([0,+\infty))$, one gets
\begin{align*}
& \sup_{x \geq 0} \Big| \int_0^{+\infty} \ppsf[N_{\lambda} \leq x] \gamma_1(\ddr x) - \int_0^{+\infty} \ppsf[N_{\lambda} \leq x] \gamma_2(\ddr x)\Big| \label{LeCamInequality} \\
&\quad\leq 2(1+e^2) \sup_{B \in \mathscr{B}([0,+\infty))} |\gamma_1(B) - \gamma_2(B)| \ . \nonumber
\end{align*}
\end{prp}
Therefore, a direct application of Proposition \ref{prop:LeCam} yields 
\begin{align*}
& \sum_{k=0}^{\infty} \Big| \int_0^{T(n,\alpha)} \big[ \ppsf[N_{\omega(t; n,\alpha)} = k] - \ppsf[N_{t n^{\alpha}} = k] \big] f_{S_{\alpha,\theta}}(t) \ddr t \\
&\quad\leq 2(1+e^2) \int_0^{T(n,\alpha)} \Big| f_{S_{\alpha,\theta}}(t) - n^{\alpha} f_{S_{\alpha,\theta}}\big(\psi_n(t n^{\alpha})\big) \psi'_n(t n^{\alpha}) \Big| \ddr t \sim \frac{1}{n^{\alpha}}  \nonumber
\end{align*}
the function $\psi_n$ denoting the inverse of $(0, T(n,\alpha)) \ni t \mapsto \omega(t; n,\alpha)$.

Finally, according to point D), first get rid of the shift $+1$ relative to the Poisson mixture. In fact, this yields an extra term which goes to zero as $1/n^{\alpha}$, since $\int_0^{1/n^{\alpha}}  f_{S_{\alpha, \theta}}(t) \ddr t \sim 1/n^{\alpha}$. Then, it remains to consider the following term
\begin{align*}
& \sup_{x \geq 0} \Big| \ppsf[N_{S_{\alpha,\theta} \cdot n^{\alpha}} \leq x n^{\alpha}] - \ppsf[S_{\alpha,\theta} \leq x] \Big| \\
&\quad= \sup_{x \geq 0} \Big| \int_0^{\infty} \left\{ \ppsf[N_{t \cdot n^{\alpha}} \leq x n^{\alpha}] - \ind\{t \leq x\} \right\} f_{S_{\alpha,\theta}}(t)\ddr t \Big| \ . \nonumber
\end{align*}
This quantity is bounded by $C_3(\alpha, \theta)/n^{\alpha}$ for a suitable $C_3(\alpha, \theta) > 0$, as a direct application of Theorem 1 of Adell and de la Cal \cite{AdCal(93)}, which is here re-stated for the reader's ease.
\begin{prp}[Adell-de la Cal] \label{prop:AdCal}
Consider a probability density $h : [0,+\infty) \rightarrow [0,+\infty)$ satisfying the following conditions
\begin{itemize}
\item[i)] $h \in \mathrm{C}_b^1([0,+\infty))$, the space of $\mathrm{C}^1$ functions on $[0,+\infty)$ which are bounded together with their first derivatives;
\item[ii)] $\sup_{x\geq 0} |h(x) + \frac12 xh'(x)| < +\infty$:
\item[iii)] there exist $b_0 \geq 0$ and $0<\gamma<\frac12$ such that
\end{itemize}
\begin{displaymath}
\lim_{s\rightarrow+\infty} \sup_{b \geq b_0}\ b \!\!\!\!\! \sup_{\substack{x \geq 0, \\ |x-b| \leq b^{1-\gamma}/s^{\gamma}}} \!\!\!\!\! | h'(x) - h'(b)| = 0\ .
\end{displaymath}
Then, there exists a constant $C(h)$, depending on the above analytical properties of $h$, for which
$$
\sup_{x \geq 0} \Big| \int_0^{\infty} \left\{ \ppsf[N_{t m} \leq x m] - \ind\{t \leq x\} \right\} h(t)\ddr t \Big| \leq \frac{C(h)}{m} 
$$
holds for all $m > 0$.
\end{prp}

To apply Proposition \ref{prop:AdCal} it is enough to set $m = n^{\alpha}$ and $h = f_{S_{\alpha, \theta}}$. Checking conditions i) and ii) is obvious, while the condition iii) can be verified as in Section 3 of the work of Adell and de la Cal \cite{AdCal(93)}. The proof Theorem \ref{be} is completed.

\begin{rmk}
According to Remark \ref{rmk_ext}, a natural extension of Theorem \ref{be} may be to consider the random number $M_{l,n}$ of blocks in the random partition $\Pi_{n}$ induced by a random sample from $\tilde{\mathfrak{p}}_{\alpha,\theta}$. In particular, Pitman \cite{Pit(06)} showed that
\begin{equation}\label{adivl}
\frac{M_{l,n}}{n^{\alpha}}\stackrel{\text{a.s.}}{\longrightarrow}\frac{\alpha[1-\alpha]_{(l-1)}}{l!} S_{\alpha,\theta},
\end{equation}
as $n\rightarrow+\infty$. Then one may apply the results in Section \ref{sect:laplace} and Section \ref{sect:poisson} to obtain a Berry-Esseen theorem for \eqref{adivl}. A more challenging task is to prove Theorem \ref{be} for the class of $\alpha$-stable Poisson-Kingman processes (Pitman \cite{Pit(06)}). 
\end{rmk}

\section{Pitman's posterior $\alpha$-diversity}

Pitman's posterior $\alpha$-diversity \eqref{post_adiv} was first introduced in Favaro et al. \cite{Fav(09)} for Bayesian nonparametric inference of the number of new species in $m$ additional samples, given $n$ initial observed samples. Samples are modeled by the random vectors $(X_1, \dots, X_n)$ and $(X_{n+1}, \dots, X_{n+m})$, which are subsequent segments of the exchangeable sequence $\{X_i\}_{i \geq 1}$ with $\tilde{\mathfrak{p}}_{\alpha,\theta}$ as directing measure. Let $K_m^{(n)} = K_m^{(n)}(X_1, \dots, X_n; X_{n+1}, \dots, X_{n+m}) := \sum_{1\leq i\leq m} \ind\{X_{n+i} \not\in \{X_1,\dots, X_n\}\}$ be the number of new species in $(X_{n+1}, \dots, X_{n+m})$. Lijoi et al. \cite{Lij(07)} obtained the posterior distribution of $K_{m}^{(n)}$ given $(X_{1},\ldots,X_{n})$. For $(X_{1},\ldots,X_{n})=(x_1, \dots, x_n) \in \N^{n}$, $K_{n}=j\in\{1,\ldots,n\}$ such that $\{X_1 = x_1,\ldots, X_n = x_n\} \cap \{K_n=j\} \neq \emptyset$, and $\mathbf{N}_{n}=(n_{1},\ldots,n_{j})\in\N^{j}$ such that $\sum_{1\leq i\leq j}n_{i}=n$, they showed that
\begin{align}\label{post_k}
& \ppsf[K_m^{(n)}=k\,|\,X_1 = x_1,\ldots,X_n = x_n, K_n=j]\\
&\notag\quad=\ppsf[K_m^{(n)}=k\,|\,K_{n}=j, \mathbf{N}_n =(n_1,\ldots,n_j)]\\
&\notag\quad=\ppsf[K_m^{(n)}=k\,|\,K_{n}=j]=\frac{\left[\frac{\theta}{\alpha}+j\right]_{(k)}}{[\theta+n]_{(m)}}\mathscr{C}(m,k;\alpha,-n+j\alpha)
\end{align}
for any $k\in\{0,1,\ldots,m\}$. Here, $\mathscr{C}(n,k;s,r)$ denotes the non-central generalized factorial coefficient, i.e., $\mathscr{C}(n,k;s,r)=\frac{1}{k!}\sum_{i=0}^k(-1)^{i}{k\choose i}[-is-r]_{(m)}$ (Charalambides \cite{Cha(02)}) The posterior distribution \eqref{post_k} is at the basis of Bayesian nonparametric inference for $K_m^{(n)}$, e.g. estimation and uncertainty quantification (Lijoi et al. \cite{Lij(07)}). However, since the computational burden for evaluating the non-central generalized factorial coefficient becomes overwhelming for large $m$, the evaluation of \eqref{post_k} is practically impossible for large $m$. To overcome this drawback, Favaro et al. \cite{Fav(09)} introduced Pitman's posterior $\alpha$-diversity $S_{\alpha,\theta}(n,j)$ to provide a practical tool to obtain large $m$ approximated posterior inferences for $K_{m}^{(n)}$ via straightforward Monte Carlo sampling from $S_{\alpha,\theta}(n,j)$. Here, we formulate a Berry-Esseen theorem for Pitman's posterior $\alpha$-diversity, thus quantifying the error of the approximated inference.

We start with a novel representation of the posterior distribution \eqref{post_k} in terms of the distribution of a  compound sum of independent Bernoulli random variables with random parameter. For any $n\in\N$ and $p\in[0,1]$, we denote by $Z(n,p)$ a 
Binomial random variable with parameters $n,p$. Also, we use the symbol $B_{a,b}$ to denote a Beta random variable with parameters $a,b > 0$.
\begin{lem}\label{lem_sum}
Let $n,m\in\N$ and $j\in \{1,\dots,n\}$. For any $\alpha\in(0,1)$ and $\theta> -\alpha$, let $K^{\ast}_{m}$ be the number of blocks of the random partition $\Pi_{m}$ induced by a random sample $(X_{1}^{\ast},\ldots,X_{m}^{\ast})$ from $\tilde{\mathfrak{p}}_{\alpha,\theta+n}$. Then, for any $k\in\{0,1,\ldots,m\}$, there holds
\begin{displaymath}
\ppsf[K_m^{(n)} = k\,|\, K_n=j] = \ppsf[Z(K_{m}^{\ast},B_{\theta/\alpha+j,n/\alpha-j})=k],
\end{displaymath}
where the random variables $Z$, $K_{m}^{\ast}$ and $B_{\theta/\alpha+j,n/\alpha-j}$ are mutually independent under $\ppsf$.
\end{lem}

\begin{proof}
Let $[x]_{[n,a]}$ be the falling factorial of $x$ of order $n$ and decrement $a$, i.e. $[x]_{[n,a]}=\prod_{0\leq i\leq n-1}(x-ia)$, let $S(n,k)$ be the Stirling number of the second kind, and let $S(n,k;x) := \sum_{k\leq t\leq n}{t\choose i}[x]_{[i-k,1]}S(n,i)$ be the non-central Stirling number of the second kind. Furthermore, let $s(n,k)$ be the Stirling number of the first kind and recall that $[x]_{[n,1]}=\sum_{i=1}^n s(n,i)x^i$. We combine the definition of $S(n,k;x)$ with Proposition 1 in Favaro et al. \cite{Fav(09)} to write, for any $r \in \N$,
\begin{align*}
&\quad \eesf[(K_m^{(n)})^r\,|\,K_n=j]\\
&\quad=\sum_{t=0}^r S(r,t)\frac{\left[j+\frac{\theta}{\alpha}\right]_{(t,1)}}{\left[\frac{\theta+n}{\alpha}\right]_{(t)}}
\left[\frac{\theta+n}{\alpha}\right]_{(t)}\sum_{i=0}^t (-1)^{t-i}{t\choose i}\frac{[\theta+n+i\alpha]_{(m)}}{[\theta+n]_{(m)}}\\
&\quad=\sum_{t=0}^r S(r,t)\frac{\left[j+\frac{\theta}{\alpha}\right]_{(t,1)}}{\left[\frac{\theta+n}{\alpha}\right]_{(t)}}
\eesf[[K^{\ast}_{m}]_{[t,1]}]\\
&\quad=\sum_{t=0}^r S(r,t)\eesf[[K^{\ast}_{m}]_{[t,1]}]
\frac{\Gamma\left(\frac{\theta+n}{\alpha}\right)}{\Gamma\left(\frac{\theta}{\alpha}+j\right)\Gamma\left(\frac{n}{\alpha}-j\right)}
\int_{0}^{1}x^{t+\frac{\theta}{\alpha}+j-1}(1-x)^{\frac{n}{\alpha}-j-1}\ddr x\\
&\quad=\sum_{t=0}^{r}S(r,t)\eesf[[K^{\ast}_m]_{[t,1]}]\eesf[(B_{\theta/\alpha+j,n/\alpha-j})^t ]\\
&\quad=\eesf\left[\left(Z(K_{m}^{\ast},B_{\theta/\alpha+j,n/\alpha-j})\right)^r \right],
\end{align*}
where the last identity follows by $\eesf[(Z(n,p))^{r}]=\sum_{0\leq t\leq r}S(r,t)[n]_{[t,1]}p^{t}$. Then the proof is completed
by resorting to the one-to-one correspondence between the conditional law of $K_m^{(n)}$ and the sequence of its conditional moments.
\end{proof}

Let $\ffsf_m(n,j)$ and $\ffsf_{\alpha,\theta}(n,j)$ stand for the distribution functions of $K^{(n)}_{m}/m^{\alpha}$ and $S_{\alpha,\theta}(n,j)$, respectively, conditioned on the event 
$\{K_n = j\}$. Then, the followiong theorem may be interpreted as the natural posterior counterpart of Theorem \ref{be}, namely a Berry-Esseen theorem for Pitman's posterior $\alpha$-diversity.

\begin{thm}\label{be_post}
Let $n \in \N$ and $j\in \{1,\dots,n\}$. For any $\theta>0$ and $\alpha\in(0,1)$ such that $\frac{n}{\alpha}-j\geq 1$, there exists a positive constant $C_{\alpha, \theta}(n,j)$, depending solely on $n$, $j$, $\alpha$ and $\theta$, such that $\ddr_K(\ffsf_{m}(n,j); \ffsf_{\alpha,\theta}(n,j)) \leq m^{-\alpha}C_{\alpha, \theta}(n,j)$ holds for every $m \in \N$.
\end{thm}

\begin{proof}
For a generic random variable $X$, let $\ffsf_{X}$ denote its distribution function. Then, notice that $\ffsf_{\alpha,\theta}(n,j)(x) = 
\eesf[F_{B_{\theta/\alpha+j,n/\alpha-j}}(x/S_{\alpha,\theta+n})]$. Hereafter, it will be shown there exists a suitable constant $C_{\alpha,\theta}^{\ast}(n,j)$ such that
\begin{equation} \label{eq:BoundKmnPost}
\ddr_K(\ffsf_{m}(n,j); \ffsf_{\alpha,\theta}(n,j)) \leq C_{\alpha,\theta}^{\ast}(n,j)\eesf\left[ \frac{1}{K^{\ast}_{m}+1} \right] + \ddr_{K}(\ffsf_{K_{m}^{\ast}/m^{\alpha}};\ffsf_{S_{\alpha,\theta+n}}),
\end{equation}
is valid, along with
\begin{equation}\label{second_bound}
\eesf\left[\frac{1}{K^{\ast}_{m}+1} \right] \leq \frac{1}{m^{\alpha}}\eesf\left[ \frac{1}{S_{\alpha,\theta+n}}\right] + \ddr_{K}(\ffsf_{K^{\ast}_{m}/m^{\alpha}}; \ffsf_{S_{\alpha,\theta+n}}).
\end{equation}
With regards to \eqref{eq:BoundKmnPost}, write
\begin{align}\label{dis1}
&\notag \ddr_K(\ffsf_{m}(n,j); \ffsf_{\alpha,\theta}(n,j))
=\sup_{x\geq0}\left|\ffsf_{K_{m}^{(n)}}(m^{\alpha} x)-\eesf\left[\ffsf_{B_{\theta/\alpha+j,n/\alpha-j}}\left(\frac{x}{S_{\alpha,\theta+n}}\right)\right]\right|\\
&\notag\quad \leq \sup_{x\geq0}\left|\ffsf_{K_{m}^{(n)}}(m^{\alpha} x)-\eesf\left[\ffsf_{B_{\theta/\alpha+j,n/\alpha-j}}\left(\frac{m^{\alpha}x}{K^{\ast}_{m}}\right)\right]\right|\\
&\notag\quad\quad+\sup_{x\geq0}\left|\eesf\left[\ffsf_{B_{\theta/\alpha+j,n/\alpha-j}}\left(\frac{m^{\alpha}x}{K^{\ast}_{m}}\right)\right]-\eesf\left[\ffsf_{B_{\theta/\alpha+j,n/\alpha-j}}\left(\frac{x}{S_{\alpha,\theta+n}}\right)\right]\right|\\
&\quad =\sup_{x\geq0}\left|\ffsf_{K_{m}^{(n)}}(x)-\eesf\left[\ffsf_{B_{\theta/\alpha+j,n/\alpha-j}}\left(\frac{x}{K^{\ast}_{m}}\right)\right]\right|\\
&\notag\quad\quad+\sup_{x\geq0}\left|\eesf\left[\ffsf_{B_{\theta/\alpha+j,n/\alpha-j}}\left(\frac{m^{\alpha}x}{K^{\ast}_{m}}\right)\right]-\eesf\left[\ffsf_{B_{\theta/\alpha+j,n/\alpha-j}}\left(\frac{x}{S_{\alpha,\theta+n}}\right)\right]\right|
\end{align}
and we treat separately the terms in \eqref{dis1}. With regard to the first term,
\begin{align*}
&\sup_{x\geq0}\left|\ffsf_{K_m^{(n)}}(x)-\eesf\left[\ffsf_{B_{\theta/\alpha+j,n/\alpha-j}}\left(\frac{x}{K^{\ast}_{m}}\right) \right]\right| \\
&\quad\leq \eesf\left[ \sup_{x\geq0}\left|\ffsf_{K_{m}^{(n)}}(x)-\ffsf_{B_{\theta/\alpha+j,n/\alpha-j}}\left(\frac{x}{K^{\ast}_{m}} \right)\right|\right] \\
&\quad= \eesf\left[ \sup_{x\geq0}|\ffsf_{K_{m}^{(n)}}(K^{\ast}_{m}x)-\ffsf_{B_{\theta/\alpha+j,n/\alpha-j}}(x)| \right] \\
&\quad= \eesf\Big[ \sup_{x\geq0} \Big| \ppsf\Big[ \frac{1}{K^{\ast}_{m}} \sum_{i=1}^{K^{\ast}_{m}} Y_i \leq x\ \big|\ K^{\ast}_{m} \Big] - \ffsf_{B_{\theta/\alpha+j,n/\alpha-j}}(x) \Big| \Big],
\end{align*}
where in the last identity we used Lemma \ref{lem_sum}. Thus, given $B_{\theta/\alpha+j,n/\alpha-j}$, the $Y_i$'s are (conditionally) independent and identically distributed Bernoulli variables with parameter $B_{\theta/\alpha+j,n/\alpha-j}$. Moreover, $K^{\ast}_{m}$, $B_{\theta/\alpha+j,n/\alpha-j}$ and the sequence $\{Y_i\}_{i\geq 1}$ are mutually independent under $\ppsf$. Since $\frac{\theta}{\alpha}+j \geq 1$ and $\frac{n}{\alpha}-j\geq 1$, invoke Corollary 3 in Dolera and Favaro \cite{Dol(18)} to conclude that
\begin{equation} \label{DFdeFinetti}
\sup_{x\geq0} \Big|\ppsf\Big[\frac{1}{K^{\ast}_{m}} \sum_{i=1}^{K^{\ast}_{m}} Y_i \leq x\ |\ K^{\ast}_{m} \Big]- \ffsf_{B_{\theta/\alpha+j,n/\alpha-j}}(x) \Big| 
\leq \frac{C_{\alpha,\theta}^{\ast}(n,j)}{K^{\ast}_{m}+1}\ ,
\end{equation}
with $C_{\alpha,\theta}^{\ast}(n,j) := \frac 12\sup_{x\in[0,1]} |\ffsf_{B_{\theta/\alpha+j,n/\alpha-j}}^{''}(x)|$, exploiting that $1/K^{\ast}_m \leq 2/(K^{\ast}_{m}+1)$. Now,
taking the expectation of both sides of \eqref{DFdeFinetti} yields the first term on the right-hand side of \eqref{eq:BoundKmnPost}. With regard to the second term in \eqref{dis1},
we notice that, for any random variable $X$ supported in $(0,+\infty)$ there holds
$\eesf[\ffsf_{B_{\theta/\alpha+j,n/\alpha-j}}\left(x/X\right) ] = \int_0^1 \ffsf_{X}(x/t) \ddr \ffsf_{B_{\theta/\alpha+j,n/\alpha-j}}(t)$. Whence,
\begin{align*}
&\sup_{x\geq0}\left|\eesf\left[\ffsf_{B_{\theta/\alpha+j,n/\alpha-j}}\left(\frac{m^{\alpha} x}{K^{\ast}_{m}}\right)\right]-\eesf\left[\ffsf_{B_{\theta/\alpha+j,n/\alpha-j}}\left(\frac{x}{S_{\alpha,\theta+n}}\right)\right]\right|\\
&\quad \leq \int_0^1 \sup_{x\geq0}\left|\ffsf_{K^{\ast}_{m}/m^{\alpha}}\left(\frac{x}{t}\right)- \ffsf_{S_{\alpha,\theta+n}}\left(\frac{x}{t}\right)\right|\ddr \ffsf_{B_{\theta/\alpha+j,n/\alpha-j}}(t)\\
&\quad\leq \ddr_{K}(\ffsf_{K^{\ast}_{m}/m^{\alpha}}; \ffsf_{S_{\alpha,\theta+n}}),
\end{align*}
yielding the second term on the RHS of \eqref{eq:BoundKmnPost}. With regard to \eqref{second_bound},
$$
\eesf\left[\frac{1}{K^{\ast}_{m}+1}\right]
\leq \left| \eesf\left[\frac{1}{K^{\ast}_{m}+1}\right] - \eesf\left[\frac{1}{m^{\alpha} S_{\alpha,\theta+n}+1}\right] \right| + \frac{1}{m^{\alpha}}\eesf\left[\frac{1}{S_{\alpha,\theta+n}}\right]
$$
and we deal separately with the two expectations inside the modulus, taking advantage that we have expressions of the form $\eesf[1/(1+X)]$. Indeed, setting $\varphi(x) = (1+x)^{-1}$, we have $\int_0^{+\infty} \varphi(x) \ddr \ffsf(x) = \varphi(0) + \int_0^{+\infty} \varphi^{'}(x)[1 - \ffsf(x)]\ddr x$ for every distribution function $\ffsf$ supported in $[0,+\infty)$. From $\int_0^{+\infty} 
|\varphi^{'}(x)| \ddr x = 1$, we get
\begin{align*}
& \left| \eesf\left[\frac{1}{K^{\ast}_{m}+1}\right] - \eesf\left[\frac{1}{m^{\alpha} S_{\alpha,\theta+n}+1}\right] \right| \\
&\quad\leq \int_0^{+\infty} |\varphi^{'}(x)| \ddr x \cdot \sup_{x \geq 0} |\ffsf_{K^{\ast}_{m}}(x) - \ffsf_{m^{\alpha} S_{\alpha,\theta+n}}(x)| \\
&\quad= \ddr_{K}(\ffsf_{K^{\ast}_{m}}; \ffsf_{m^{\alpha} S_{\alpha,\theta+n}}) = \ddr_{K}(\ffsf_{K^{\ast}_{m}/m^{\alpha}}; \ffsf_{S_{\alpha,\theta+n}})\ .
\end{align*}
To conclude, note that $\eesf[1/S_{\alpha,\theta+n}] = 
\frac{\alpha\Gamma(\theta+n+1)\Gamma(\alpha(\theta+n-\alpha))}{\Gamma(\frac{\theta+n}{\alpha}+1)\Gamma(\theta+n-\alpha)}$
is finite for all $n \in \N$ whenever $\theta > -\alpha$, so that \eqref{second_bound} holds true. The proof is completed by combining \eqref{eq:BoundKmnPost}--\eqref{second_bound}, and then by a direct application of Theorem \ref{be}.
\end{proof}
 
As a direct consequence of Theorem \ref{be_post}, it is possible to construct credible intervals containing high posterior probability for the Bayesian nonparametric estimator 
$\hat{K}_{m}^{(n)}$. Due to the definition of the Kolmogorov distance, for every $a<b$
$$
\ppsf\Big[\frac{K_{m}^{(n)}}{m^{\alpha}} \in [a,b]\,\big|\,K_n=j\Big]
 \geq \ppsf\big[S_{\alpha,\theta}(n,j) \in [a,b]\,\big|\,K_n=j\big] - 2m^{-\alpha}C_{\alpha, \theta}(n,j)\ .
$$
Therefore, after fixing any confidence level, say $1-\gamma$, it is possible to choose $a,b$ so that
\begin{displaymath}
\left\{ \begin{array}{ll}
\text{length of $(a,b)$} \ \text{is\ minimum}\\[0.2cm]
\ppsf\big[S_{\alpha,\theta}(n,j) \in [a,b]\,\big|\,K_n=j\big] \geq 1-\gamma+2m^{-\alpha}C_{\alpha, \theta}(n,j)\ .
\end{array} \right.
\end{displaymath}
This provides with a practical, simple tool for assessing uncertainty quantification in the context of the Bayesian nonparametric estimation of $K_{m}^{(n)}$.

We conclude this section by pointing out a useful generalization of Theorem \ref{be_post}, which follows by the invariance under scaling of the Kolmogorov distance. Let $\lambda(m)$ be an arbitrary function of $m$ such that $\lambda(m)/m^{\alpha}\rightarrow1$ as $m\rightarrow+\infty$. That is, $\lambda(m)$ is asymptotically equivalent to $m^{\alpha}$ for large $m$. Moreover, let $\ffsf^{(\lambda)}_{m}(n,j)$ denote the posterior distribution function of $K^{(n)}_{m}/\lambda(m)$, given $(X_{1},\ldots,X_{n})$ featuring $K_{n}=j\leq n$ distinct species. For any $\alpha\in(0,1)$, $\theta>-\alpha$ and $n\in\N$ such that $\frac{n}{\alpha}-j\geq1$, Theorem \eqref{be_post} implies that there exists a constant $C^{(\lambda)}_{\alpha,\theta}(n,j)$, depending solely on $n$, $j$, $\alpha$ and $\theta$, such that 
\begin{equation}\label{be_post_g}
\ddr_K(\ffsf^{(\lambda)}_{m}(n,j); \ffsf_{\alpha,\theta}(n,j)) \leq \frac{C^{(\lambda)}_{\alpha, \theta}(n,j)}{\lambda(m)}
\end{equation}
for every $m \in \N$. One may apply \eqref{be_post_g} to identify a function $\lambda$ that leads to an upper bound for $\ddr_K(\ffsf^{(\lambda)}_{m}(n,j); \ffsf_{\alpha,\theta}(n,j))$ which is smaller than the upper bound of Theorem \ref{be_post}. For instance, one may consider $\lambda(m)=(c+m)^{\alpha}-c^{\alpha}$, with $c$ being a positive parameter, and then optimize the upper bound in \eqref{be_post_g} with respect to $c$.

\section*{Acknowledgements}

The authors thank an anonymous referee for all her/his comments, corrections, and suggestions which remarkably improved the paper. Emanuele Dolera and Stefano Favaro received funding from the European Research Council (ERC) under the European Union's Horizon 2020 research and innovation programme under grant agreement No 817257. Emanuele Dolera and Stefano Favaro gratefully acknowledge financial support from the Italian Ministry of Education, University and Research (MIUR), ``Dipartimenti di Eccellenza" grant 2018-2022.


\begin{flushleft}

EMANUELE DOLERA\\  
Universit\`a di Pavia\\ 
Dipartimento di Matematica\\
Via Adolfo Ferrata 5, 27100 Pavia, Italy\\
emanuele.dolera@unipv.it

\bigskip

%
STEFANO FAVARO\\
Universit\`a  degli studi di Torino\\ 
Dipartimento di Economia e Statistica\\ 
Lungo Dora Siena 100A, 10134 Torino, Italy\\ 
stefano.favaro@unito.it

\end{flushleft}


\begin{thebibliography}{26}

\bibitem{AdCal(93)}  
\textsc{Adell, J.A. and de la Cal, J.} (1993). On the uniform convergence of normalized Poisson mixtures to their mixing distribution. \textit{Statist. Probab. Lett.}, \textbf{18}, 227--232

\bibitem{Ald(85)}  
\textsc{Aldous, D.} (1985). \textit{Exchangeability and related topics.} Ecole d'Et\'e de Probabilit\'es de Saint-Flour XIII. Lecture notes in mathematics, Springer - New York.

\bibitem{Cha(02)}  
\textsc{Charalambides, C.A.} (2002). \textit{Enumerative combinatorics}. Chapman and Hall/CRC.

\bibitem{Cra(16)}  
\textsc{Crane, H.} (2016). The ubiquitous Ewens sampling formula. \textit{Statist. Sci.}, \textbf{31}, 1--19.

\bibitem{Dev(09)}
\textsc{Devroye, L.} (2009). Random variate generation for exponentially and polynomially tilted stable distributions. \textit{ACM Trans. Model. Comp. Simul.}, \textbf{19}, 4.

\bibitem{Dol(18)}  
\textsc{Dolera, E. and Favaro, S.} (2018). De Finetti's theorem: rate of convergence in Kolmogorov distance. \textit{Preprint: arXiv:1802.02244}

\bibitem{Fav(15)}
\textsc{Favaro, S. and Feng, S.} (2015). Large deviation principles for the Ewens-Pitman sampling model. \textit{Electron. J. Probab.}, \textbf{20}, 1--27.

\bibitem{Fav(18)}
\textsc{Favaro, S., Feng, S. and Gao. F.} (2018). Moderate deviations for Ewens-Pitman sampling models. \textit{Sankhya A}, to appear.

\bibitem{Fav(09)}  
\textsc{Favaro, S., Lijoi, A., Mena, R.H. and Pr\"unster, I.} (2009). Bayesian nonparametric inference for species variety with a two parameter Poisson-Dirichlet process prior. \textit{J. Roy. Statist. Soc. Ser. B}, \textbf{71}, 993--1008.

\bibitem{Fen(98)}  
\textsc{Feng, S. and Hoppe, F.M.} (1998). Large deviation principles for some random combinatorial structures in population genetics and Brownian motion. \textit{Ann. Appl. Probab.}, \textbf{8}, 975--994.

\bibitem{Hwa(99)}
\textsc{Hwang, H.K.} (1999). Asymptotics of poisson approximation to random discrete distributions: an analytic approach. \textit{Adv. Appl. Prob.}, \textbf{31}, 448--491.

\bibitem{Kan(86)}
\textsc{Kanter, M.} (1986). Stable densities under change of scale and total variation inequalities. \textit{Ann. Probab.}, \textbf{3}, 697--707.

\bibitem{Kor(73)}
\textsc{Korwar, R.M. and Hollander, M.} (1973). Contribution to the theory of Dirichlet processes. \textit{Ann. Probab.}, \textbf{1}, 705--711.

\bibitem{Kull(47)}
\textsc{Kullback, S.} (1947). On the Charlier type B series. \textit{Ann. Math. Statist.}, \textbf{18}, 574--581.

\bibitem{LeCam(55)}
\textsc{Le Cam, L.} (1955). Mixture of Poisson distributions. I. Some simple cases, \textit{Unpublished manuscript}.

\bibitem{Lij(07)}   
\textsc{Lijoi, A., Mena, R.H. and Pr\"unster, I.} (2007). Bayesian nonparametric estimation of the probability of discovering new species. \textit{Biometrika}, \textbf{94}, 769--786.

\bibitem{Lij(14)}   
\textsc{Lijoi, A., Pr\"unster, I. and Walker, S.G.} (2014). A note on ``Bayesian nonparametric estimators derived from conditional Gibbs structures''. \textit{Ann. Appl. Probab.}, \textbf{24}, 447--448.

\bibitem{Mai(10)}
\textsc{Mainardi, F., Mura, A. and Pagnini, G.} (2010). The $M$--Wright function in time-fractional diffusion processes: a tutorial survey. Int. J. Differ. Equ. Art. ID 104505, 29 pp.

\bibitem{Olv(74)}
\textsc{Olver, F.W.J.} (1974). \textit{Asymptotic and special functions}. 

\bibitem{Orl(16)}
\textsc{Orlitsky, A., Suresh, A.T. and Wu, Y.} (2016). Optimal prediction of the number of unseen species. Proc. Natl. Acad. Sci. USA \textbf{113}, 13283--13288.

\bibitem{Per(18)}
\textsc{Pereira, A., Oliveira, R.I. and Ribeiro, R.} (2018). Concentration in the generalized Chinese restaurant process. \textit{Preprint: arXiv:1806.09656v1}

\bibitem{Per(92)}
\textsc{Perman, M., Pitman, J. and Yor, M.} (1992). Size-biased sampling of Poisson point processes and excursions. \textit{Probab. Theory Related Fields}, \textbf{92}, 21--39.

\bibitem{Pit(95)}
\textsc{Pitman, J.} (1995). Exchangeable and partially exchangeable random partitions. \textit{Probab. Theory Related Fields}, \textbf{102}, 145--158.

\bibitem{Pit(06)} 
\textsc{Pitman, J.} (2006). \textit{Combinatorial Stochastic Processes.} Ecole d'Et\'e de Probabilit\'es de Saint-Flour XXXII. Lecture notes in mathematics, Springer - New York.

\bibitem{Qi(13)} 
\textsc{Qi, F. and Luo, Q.M.} (2013) Bounds for the ratio of two gamma functions: from Wendel's asymptotic relation to Elezovi\'c-Giordano-Pe\v{c}ari\'c's theorem. \textit{J. Inequal. Appl.}, 2013:542.

\bibitem{Rom(84)}  
\textsc{Roman, S.} (1984). \textit{The umbral calculus}. Academic Press.

\bibitem{Sho(78)}
\textsc{Shorgin, S.Ya.} (1978). Approximation of a generalized binomial distribution. \textit{Theory Probab. Appl.}, \textbf{22}, 846--850. 

\bibitem{Usp(31)} 
\textsc{Uspensky, J.V.} (1931) On Ch. Jordan's series for probability. \textit{Ann. of Math.}, \textbf{32}, 306--312.

\bibitem{Won(01)}
\textsc{Wong, R.} (2001). \textit{Asymptotic approximations of integrals}. SIAM: Society for Industrial and Applied Mathematics.

\bibitem{Zol(86)}
\textsc{Zolotarev, V.M.} (1986). \textit{One dimensional stable distributions}. American Mathematical Society.

\end{thebibliography}
\end{document}